\documentclass[11pt, letterpaper]{amsart}
\usepackage[plainpages=false,pdfpagelabels,
colorlinks=true,linkcolor=blue,
citecolor=blue,filecolor=blue,      
urlcolor=blue,hypertexnames=false]{hyperref}
\usepackage{amssymb}
\usepackage{graphics}
\usepackage{graphicx}               
\usepackage{latexsym}
\usepackage{amsmath}
\usepackage{amssymb,amsthm,amsfonts}
\usepackage{mathrsfs}
\usepackage{amscd}
\usepackage[arrow, matrix, curve]{xy}
\usepackage{syntonly}
\usepackage{tikz}
\usetikzlibrary{quotes}
\usepackage{tikz-cd}
\usepackage{quiver}

\usepackage{enumitem}
\usepackage{todonotes}
\usepackage{hyperref,cleveref}

\theoremstyle{plain}
\newtheorem{thm}{Theorem}[section]
\newtheorem{theorem}[thm]{Theorem}
\newtheorem*{theorem*}{Theorem}
\newtheorem{lemma}[thm]{Lemma}

\newtheorem{conj}[thm]{Conjecture}

\theoremstyle{definition}

\numberwithin{equation}{thm}

\title[Hitchin morphism]{On the image of Hitchin morphism for some classical groups on algebraic surfaces}

\author{Artan Sheshmani \textsuperscript{1,2,3}}
\email{artan@mit.edu}
\author{Jianping Wang \textsuperscript{1}}
\email{jianpw@bimsa.cn}
\author{Xiaopeng Xia \textsuperscript{1}}
\email{xpxia@mail.ustc.edu.cn}

\address{\textsuperscript{1}Beijing Institute of Mathematical Sciences and Applications, No. 544, Hefangkou Village, Huaibei Town, Huairou District, Beijing 101408}
\address{\textsuperscript{2}Massachusetts Institute of Technology, IAiFi Institute, 77 Massachusetts Ave, 26-555. Cambridge, MA 02139}
\address{\textsuperscript{3}National Research University, Higher School of Economics, Russian Federation, Laboratory of Mirror Symmetry, NRU HSE, 6 Usacheva str., Moscow, Russia, 119048}

\date{}

\begin{document}

\begin{abstract}
In this article, we study the image of the Hitchin morphism for some classical groups over an algebraic surface. The Hitchin morphism is a map from the moduli stack of $G$-Higgs bundles $\mathscr{M}_{X,G}$ to the Hitchin base $\mathscr{A}_{X,G}$, where $X$ is a smooth projective variety. In general, this morphism is not surjective when the dimension of $X$ is greater than one. Chen and Ng{\^o} showed that the Hitchin morphism factors through a closed subscheme $\mathscr{B}_{X,G}$ of the Hitchin base, which is called the spectral base. They conjectured that the image of the Hitchin morphism is exactly the spectral base. When $X$ is a smooth projective surface, we prove that this conjecture holds for the special linear algebraic group of odd rank. We also confirm this conjecture for the classical groups ${\rm SL}_n$ and ${\rm Sp}_{2n}$ when $X$ is a product of smooth curves.
\end{abstract}
\maketitle 

\tableofcontents

\section{Introduction}
Let $\mathbb{K}$ be an algebraically closed field of characteristic $0$. Let $X$ be an irreducible smooth projective variety over $\mathbb{K}$. Let $G$ be a reductive group over $\mathbb{K}$ of rank $n$. For a principal $G$-bundle $E$, we denote the adjoint bundle of $E$ by $\operatorname{ad}(E)$. A $G$-Higgs bundle is a pair $(E,\theta)$ consisting of a principal $G$-bundle $E$ and a section 
$$\theta \in H^0(X, \operatorname{ad}(E) \otimes_{\mathscr{O}_X} \Omega_X^1)$$
satisfying the integrability condition $\theta \wedge \theta = 0$, where $\Omega_X^1$ is the cotangent sheaf of $X$. Let $M_{X,G}$ be the moduli space of semistable $G$-Higgs bundles over $X$; then we have a natural morphism $h_{X,G}: M_{X,G} \rightarrow \mathscr{A}_{X,G}$ called the Hitchin morphism, where 
$$\mathscr{A}_X = \bigoplus_{i=1}^n H^0(X, S^{e_i}\Omega_X^1)$$
is the Hitchin base. The morphism $h_{X,G}$ was first constructed by Hitchin in \cite{Hit87} in the curve case, and then generalized to higher-dimensional varieties by Simpson in \cite{Sim94}. The Hitchin morphism is proper by the work of Nitsure \cite{Nit91} and  Simpson \cite{Sim94}. In the curve case, it is known that $h_{X,GL_n}$ is dominant \cite{BNR89}. As a corollary, $h_{X,GL_n}$ is surjective in the curve case.

However, in the higher-dimensional case, the Hitchin morphism is not surjective in general. In fact, to understand the image of the Hitchin morphism, in \cite{chen2020hitchin}, Chen and Ng{\^o} introduced the \emph{spectral base} $\mathscr{B}_{X,G}$, which is a closed subset of the Hitchin base $\mathscr{A}_{X,G}$, and showed that the image of the Hitchin morphism is contained in the spectral base. More precisely, we have the following commutative diagram
\[
\begin{tikzcd}
	{\mathscr{M}_{X,G}} & {\mathscr{A}_{X,G}} \\
	& {\mathscr{B}_{X,G}}
	\arrow["{h_{X,G}}", from=1-1, to=1-2]
	\arrow["{sd_{X,G}}"', from=1-1, to=2-2]
	\arrow[hook, from=2-2, to=1-2]
\end{tikzcd},
\]
where $\mathscr{M}_{X,G}$ is the stack of $G$-Higgs bundles over $X$ and $h_{X,G}: \mathscr{M}_{X,G}\rightarrow \mathscr{A}_{X,G}$ is the Hitchin morphism. The map $sd_{X,G}:\mathscr{M}_{X,G}\rightarrow \mathscr{B}_{X,G}$ is called the \emph{spectral morphism}. Chen and Ng{\^o} conjectured that the image of $h_{X,G}$ is exactly the spectral base.

\begin{conj}[{\cite[Conjecture 5.2]{chen2020hitchin}}] For every  $b\in \mathscr{B}_{X,G}(\mathbb{K})$, the fibre $h_{X,G}^{-1}(b)$ is nonempty.
\end{conj}
 The conjecture remains open even for $G={\rm GL}_n$; only some special cases are known: the above conjecture holds when 
\begin{enumerate}
\item[(i)] $\dim X = 2$, $G = {\rm GL}_n$ with $n$ arbitrary \cite{song2023image}, 
\item[(ii)] $\dim X$ arbitrary, and $G = {\rm GL}_2$ \cite{HL24},
\item[(iii)] $X$ is a $d$-dimensional Abelian variety, and $G$ is a reductive group \cite{chen2020hitchin},
\item[(iv)] $X$ is a ruled surface or nonisotrivial elliptic fibration with only reduced fibers, and $G$ is a reductive group \cite{huynh2026hitchin},
\item[(v)] $X$ is a hyperelliptic variety in characteristic 0, and $G = {\rm GL}_n$ with $n$ arbitrary \cite{patel2026stratifying},
\item[(vi)] $X$ is a smooth projective variety whose canonical divisor $K_X$ is numerically trivial, and $G = {\rm GL}_n$ with $n$ arbitrary \cite{patel2026hitchin}.
\end{enumerate}

The aim of this paper is to verify the conjecture for some other classical algebraic groups when $\dim X = 2$. Our main result is the following:

\begin{theorem}[{ Theorem \ref{thm: sd_SL(2n+1) surj}}, Theorem  \ref{thm: sd_SL(n) surj for X=C1 times C2} and Theorem \ref{thm: sd_Sp(2n) surj for X=C1 times C2}]
Let $X$ be a smooth projective surface. Then the spectral morphism $sd_{X,G}:\mathscr{M}_{X,G}\rightarrow \mathscr{B}_{X,G}$ is surjective in the following cases:
\begin{enumerate}
\item $G = {\rm SL}_n$ with $n$ odd;
\item $X = C_1 \times C_2$ and $G = {\rm SL}_{n}$ with $n$ arbitrary, where $C_i$ is a smooth projective curve;
\item  $X = C_1 \times C_2$ and $G = {\rm Sp}_{2n}$ where $C_i$ is a smooth projective curve.
\end{enumerate}
\end{theorem}

To prove the main results, we consider the natural faithful representations $\rho:G\hookrightarrow\operatorname{GL}(V)$ for the classical groups $G$. The main strategy is to relate the spectral base $\mathscr{B}_{X,G}$ to the spectral base $\mathscr{B}_{X,\operatorname{GL}(V)}$. More precisely, for $G=\operatorname{SL}_n$ we consider the embedding $\operatorname{SL}_n\hookrightarrow\operatorname{GL}_n$. Then an $\operatorname{SL}_n$-Higgs bundle is equivalent to a $\operatorname{GL}_n$-Higgs bundle $(E,\theta)$ over $X$ such that $\det(E)\cong\mathscr{O}_X$ and $\operatorname{tr}(\theta)=0$. There is also a closed embedding of spectral bases $\mathscr{B}_{X,\operatorname{SL}_n}\hookrightarrow\mathscr{B}_{X,\operatorname{GL}_n}=\operatorname{Sect}(X,\operatorname{Sym}^n(T_X^*/X))$. Let $b\in\mathscr{B}_{X,\operatorname{SL}_n}$ be a spectral datum. By the decomposition of the spectral data (Lemma~\ref{lem: decomposition in SS}), there exist $b_i:X\to\operatorname{Sym}^{\alpha_i}(T_X^*/X)$, $1\leq i\leq s$, with $b_i(\eta)\in\operatorname{Sym}_{(1^{\alpha_i})}^{\alpha_i}(T_X^*/X)$ for $1\leq i\leq s$, where $\eta$ is the generic point of $X$, such that $b$ is the following composition of morphisms:
\[
X\xrightarrow{(b_1,\dots,b_s)} \operatorname{Sym}^{\alpha_1}(T_X^*/X)\times_X\cdots\times_X\operatorname{Sym}^{\alpha_s}(T_X^*/X)\xrightarrow{\tau_{\mu}} \operatorname{Sym}^n(T_X^*/X).
\]
Let $p_{b_i}:X_{b_i}\to X$ be the spectral cover (Lemma~\ref{lem: X_b of ChenNgo}). Let $X_1,\dots,X_r$ be the irreducible components of $X_{b_i}$. For any $1\leq j\leq r$, let $\pi_j':Y_j\to X_j$ be the normalization morphism, $\pi_j=p_{b_i}\circ\pi_j':Y_j\to X$, and $\pi_{b_i}:Y_{b_i}=\bigsqcup_{1\leq j\leq r}Y_j\to X_{b_i}\to X$ the natural morphisms. Let
\[
(E,\theta)=\bigoplus_{1\leq i\leq s}\bigl((E_{i1},\theta_{i1})\oplus(E_{i2},\theta_{i2})\oplus\cdots\oplus(E_{i\mu_i},\theta_{i\mu_i})\bigr),
\]
where $E_{ij}=\pi_{b_i*}(\mathcal{O}_{Y_{b_i}}(D_{i,j}))$ and $D_{ij}$ is a cycle of codimension $1$ on $Y_{b_i}$. Then $(E,\theta)$ is a $\operatorname{GL}_n$-Higgs bundle with spectral datum $b$. Moreover, by Lemma~\ref{fomular}, the determinant of $E_{ij}$ is also known. When $n$ is odd, we can choose $D_{ij}$ and twist by suitable line bundles with $E_{ij}$ to obtain a determinant-trivial bundle $E'$. This produces an $\operatorname{SL}_n$-Higgs bundle $(E',\theta')$ with spectral datum $b$.

When $X=C_1\times C_2$, we have the following morphism of Hitchin bases:
\[
\mathscr{A}_{X,\operatorname{GL}_n}\cong\bigoplus_{i=1}^n H^0(X,S^i\Omega_X^1)\hookrightarrow
\Bigl(\bigoplus_i H^0(C_1,S^i\Omega_{C_1}^1)\Bigr)\otimes
\Bigl(\bigoplus_i H^0(C_2,S^i\Omega_{C_2}^1)\Bigr).
\]
Given $b\in\mathscr{B}_{X,\operatorname{SL}_n}$ with $b(\eta)\in\operatorname{Sym}_{(1^{n})}^{n}(T_X^*/X)$, its image under this map yields spectral data $b_1$ on $C_1$ and $b_2$ on $C_2$. Their spectral covers $C_{1,b_1}$ and $C_{2,b_2}$ are curves. A key observation (Lemma~\ref{lem: X_b to C_1bC_2b}) is that the surface spectral cover $X_b$ embeds into the product $C_{1,b_1}\times C_{2,b_2}$ as a union of some irreducible components. As a consequence, the normalization $Y_j$ of each component of $X_b$ is again a product of smooth curves (Lemma~\ref{lem: Y_i=C times C' sm}). Based on this simple description of the normalization, a twist technique as before and the decomposition of the spectral data, we can prove the conjecture for $G=\operatorname{SL}_n$ with arbitrary $n$ when $X$ is a product of smooth projective curves.

When $G=\operatorname{Sp}_{2n}$, consider the natural faithful representation $\operatorname{Sp}_{2n}\hookrightarrow\operatorname{GL}_{2n}$. Then an $\operatorname{Sp}_{2n}$-Higgs bundle on $X$ is equivalent to $(\mathscr{E},\theta,Q)$, where $(\mathscr{E},\theta)$ is a Higgs bundle of rank $2n$ on $X$, and $Q:\mathscr{E}\otimes\mathscr{E}\to\mathscr{O}_X$ is a perfect alternating bilinear form such that $\theta$ is anti-self-adjoint with respect to $Q$. Let $b\in\mathscr{B}_{X,\operatorname{Sp}_{2n}}\subset\mathscr{B}_{X,\operatorname{GL}_{2n}}$ with $b(\eta)\in\operatorname{Sym}_{(1^{2n})}^{2n}(T_X^*/X)$. The spectral cover $X_b$ admits a natural involution $\tau$. Based on this involution, the simple description of the normalization of the components of the spectral cover in the case $X=C_1\times C_2$ and the decomposition of the spectral data, we can prove the surjectivity of the spectral morphism.

We briefly describe the content of the article. In Section~2, we collect notations and preliminary facts on Cohen--Macaulay modules, determinant formulas, dualizing sheaves, and the Hitchin/spectral data morphisms. In particular, we prove the determinant formula (Lemma~\ref{fomular}), which is our technical lemma for establishing the surjectivity of the spectral morphism for \(G=\mathrm{SL}_n\). Section~3 treats the case \(G=\mathrm{SL}_n\) with \(2\nmid n\). Section~4 is devoted to the case \(X=C_1\times C_2\), where each \(C_i\) is a smooth projective curve: Subsection~4.1 handles \(\mathrm{SL}_n\) (any \(n\)), and Subsection~4.2 handles \(\mathrm{Sp}_{2n}\).

\subsection*{Acknowledgments}
The first author is supported by grants from Beijing Institute of Mathematical Sciences and Applications (BIMSA), the Beijing NSF BJNSF-IS24005, and the China National Science Foundation (NSFC) NSFC-RFIS program W2432008. He would also like to thank China's National Program of Overseas High Level Talent for generous support.
Finally, he would like to thank NSF AI Institute for Artificial Intelligence and Fundamental Interactions at Massachusetts Institute of Technology (MIT) which is funded by the US NSF grant under Cooperative Agreement PHY-2019786. 

The authors would like to thank Matthew Huynh for sharing some references.

\section{Notations and preliminaries }

In this section we fix some notations and record some useful lemmas that will be used frequently in the subsequent sections.

\subsection{Notations}\label{subsection: notations}

Throughout the paper, $\mathbb{K}$ is an algebraically closed field with $\mathrm{char}\mathbb{K}=0$, $X$ is a projective smooth irreducible variety over $\mathbb{K}$ of dimension $2$, $G$ is a reductive group over $\mathbb{K}$ with Lie algebra $\mathfrak{g}$, $T$ is a maximal torus of $G$ with Lie algebra $\mathfrak{t}$, $W:=N_G(T)/T$ is the Weyl group. 

If $Y$ is a scheme with a $G$-action, $S$ is
a scheme, and $P$ is a $G$-principal bundle over $S$, then we denote the associated
bundle by $P\times_G Y$. When $G$ is the general linear group $GL_n$, and $E$ is a geometric vector bundle over $S$, then we have the associated
bundle $Fr(E)\times_G Y$, where $Fr(E)$ is the frame bundle of $E$.

We denote by $T_X^*$ the cotangent bundle over $X$, and by $\mathscr{T}_X$ the tangent sheaf on $X$. For $S$-scheme $Y$, we denote by $Sym^n(Y/S)$ the scheme $Y\times_S Y\times_S \cdots Y/\mathfrak{S}_n$.

For $1\leq i\leq n$, we denote by $\sigma_i$ the $i$-th elementary symmetric polynomial in  $\mathbb{K}[x_1,\dots,x_n]$.
Then $\mathbb{A}^n/\mathfrak{S}_n=\mathrm{Spec}\mathbb{K}[\sigma_1,\dots,\sigma_n]$.
For $1\leq i\leq n$ and $i_1,i_2\geq 0,i_1+i_2=i$, we denote by $\sigma_{i_1i_2}$ in $\mathbb{K}[x_{r1},x_{r2}\mid 1\leq r\leq n]$, such that
\[
\sigma_i(t_1x_{11}+t_2x_{12},\dots,t_1x_{n1}+t_2x_{n2})=\sum_{i_1+i_2=i}\sigma_{i_1 i_2}t_{1}^{i_1}t_2^{i_2}.
\]
Then $\mathbb{K}[x_{r1},x_{r2}\mid 1\leq r\leq n]^{\mathfrak{S}_n}=\mathbb{K}[\sigma_{i_1i_2}\mid 1\leq i_1+i_2\leq n]$.

\subsection{Cohen\textendash Macaulay modules and reflexive sheaves}

In this section, we collect some basic facts
about reflexive sheaves and Cohen\textendash Macaulay modules, and we refer the reader to \cite{serre2000local} and \cite{hartshorne1980stable} for a more detailed discussion.

Let $R$ be a commutative ring and let $M$ be a $R$-module. The set $\{\mathfrak{p}\in \mathrm{Spec}(R)| M_{\mathfrak{p}}\neq 0\}$ is called the support of $M$, and written $\mathrm{Supp}(M)$. Define $\dim(M) = \dim(\mathrm{Supp}(M))$.

A sequence $a_1,\dots,a_r$ of elements of $R$ is called a regular sequence for $M$ if $a_i$ is not a zero divisor in $M/(a_0,\dots,a_{i-1})M$ where $a_0=0$ for all $i=1,\dots,r$ and $M/(a_1,\dots,a_{r})M\neq 0$. If $R$ is a local ring with maximal ideal $\mathfrak{m}$, then the depth of $M$ is the maximum length of a regular sequence $a_1,\dots,a_r$ for $M$ with all $a_i\in \mathfrak{m}$.

Let $R$ be a Noetherian local ring. A finite $R$-module $M$ is Cohen\textendash Macaulay if $\mathrm{depth}(M) = \dim(M)$. A maximal Cohen\textendash Macaulay module is a Cohen\textendash Macaulay module $M$ such that $\dim(M) = \dim(R)$. 

In general, if $R$ is a Noetherian ring, then $M$ is a Cohen\textendash Macaulay module if $M_{\mathfrak{m}}$ is a Cohen\textendash Macaulay module for all maximal ideals $m\in \mathrm{Supp}(M)$. However, for $M$ to be a maximal Cohen\textendash Macaulay module, we require that $M_{\mathfrak{m}}$ is such a $R_{\mathfrak{m}}$-module for each maximal ideal $\mathfrak{m}$ of $R$.
A Noetherian ring $R$ is called Cohen\textendash Macaulay if it is Cohen\textendash Macaulay as a $R$-module.

We need the following useful facts.

\begin{lemma}[IV, B, Proposition 11 in \cite{serre2000local}]\label{lemma: A to B, E CM}
Let $A$ and $B$ be two noetherian local rings and let $\phi:A\to B$ be a homomorphism which makes $B$ into a finitely generated
$A$-module. If $E$ is a finitely generated $B$-module, then $E$ is a Cohen\textendash Macaulay $A$-module if and only if it is a Cohen\textendash Macaulay $B$-module.
\end{lemma}

\begin{lemma}[IV, D, Corollary 2 in \cite{serre2000local}]\label{lemma: CM iff free}
Assume $A$ is a regular noetherian local ring. Let $M$ be a finitely generated $A$-module and $\dim(M)=\dim(A)$. Then $M$ is Cohen\textendash Macaulay if and only if it is free.
\end{lemma}

Let $\mathscr{F}$ be a coherent sheaf on an integral Noetherian scheme $Y$. The dual of $\mathscr{F}^{\vee}$ is the sheaf $\mathscr{H}om_{\mathscr{O}_Y}  (\mathscr{F},\mathscr{O}_Y)$.
If the natural map $\mathscr{F}\to \mathscr{F}^{\vee\vee}$ to its double dual is an isomorphism, we say that $\mathscr{F}$ is reflexive.

\begin{lemma}[Proposition 1.3 in \cite{hartshorne1980stable}]\label{lemma: reflexive to (S_2)}
 Let $Y$ be a normal integral Noetherian scheme. Let $\mathscr{F}$ be a coherent reflexive $\mathscr{O}_Y$-module. Then $\mathscr{F}$ is $(S_2)$.
\end{lemma}

\subsection{A determinant formula}

Let $X$ be a smooth irreducible algebraic surface over $\mathbb{K}$.
Let $\pi:Y\to X$ be a finite flat morphism and $Y=\bigsqcup_{1\leq i\leq r} Y_i$ and $Y_i$ be a normal irreducible surface.

Let $Z^1(Y)$ be the group of cycles of codimension 1 on $Y$.
For any cycle $D\in Z^1(Y_i)$, then $D$ is a Weil divisor, then we can define $\mathscr{O}_{Y_i}(D)$.
For any open subset $U$ of $Y_i$,
\[
\mathscr{O}_{Y_i}(D)(U)=\{f\in K(Y_i)^*\mid (\mathrm{div}(f)+D)|_U\geq 0\}\cup \{0\}.
\]
Note that $\mathscr{O}_{Y_i}(D)$ is a reflexive $\mathscr{O}_{Y_i}$-module.
For any cycle $D\in Z^1(Y)$, there are unique $D_i\in Z^1(Y_i)$ for $1\leq i\leq r$ such that $D=\sum_{i=1}^rD_i$, we can define 
\[
\mathscr{O}_{Y}(D)=\bigoplus_{i=1}^r{j_i}_*\mathscr{O}_{Y_i}(D_i)
\]
where $j_i:Y_i\to Y$ is the natural morphism.

We can define $\pi_*: Z^1(Y)\to Z^1(X)$, such that $\pi_*[V]=[K(V):K(\pi(V))]\pi_*(V)$ for any prime cycle $[V]$ on $Y$.

\begin{lemma}\label{fomular}
For any cycle $D$ on $Y$, $\pi_*\mathscr{O}_{Y}(D)$ is a locally free sheaf and 
\[
\det(\pi_*\mathscr{O}_Y(D))=\det(\pi_*\mathscr{O}_Y)\otimes \mathscr{O}_X(\pi_*(D)).
\]
\end{lemma}
\begin{proof}
Since $\mathscr{O}_{Y}(D)|_{Y_i}$ is a reflexive $\mathscr{O}_{Y_i}$-module and $Y_i$ is normal, then $\mathscr{O}_{Y}(D)|_{Y_i}$ is Cohen\textendash Macaulay by Lemma \ref{lemma: reflexive to (S_2)}. Note that $X$ is regular and we have $$\dim(\mathrm{Supp}(\mathscr{O}_{Y}(D)|_{Y_i}))=2=\dim(X).$$ By Lemma \ref{lemma: CM iff free}, $\pi_*\mathscr{O}_{Y}(D)$ is locally free sheaf.

Let $D=\sum_i n_iD_i$, where the $D_i$ are distinct prime cycles on $Y$.
The proof will proceed by induction on $n=\sum_i|n_i|$.
If $n=0$ and $D=0$, we have $\det(\pi_*\mathscr{O}_Y(0))=\det(\pi_*\mathscr{O}_Y)\otimes \mathscr{O}_X(\pi_*(0))=\det(\pi_*\mathscr{O}_Y)$. 
Assume that the formula is true for $\sum_i\lvert n_i\rvert<n$.
We can assume that $n_1<0$, then we set $D'=D+D_1$. Thus, $D'=\sum_in_i'D_i$ where
$n_i'=n_i$ if $i\neq 1$, and $n_1'=n_1+1$, so that $\sum_i\lvert n_i'\rvert=\sum_i\vert n_i\rvert-1$. 
Take an open subset $U$ of $X$ such that $\mathrm{codim}(X\setminus U)=2$, and $Y'=\pi^{-1}(U)$, $B_1=D_1|_{Y'}$ and $V=\pi(D_1)\cap U$ are regular.
We only need to show that $\det(\pi'_*\mathscr{O}_{Y'}(B))=\det(\pi'_*\mathscr{O}_{Y'})\otimes \mathscr{O}_U(\pi'_*(B))$ where $\pi'=\pi|_{Y'}$ and $B=D|_{Y'}$. Let $B'=D'|_{Y'}$. Then $B,B',B_1$ are Cartier divisors on $Y'$.
Consider the exact sequence
\[
0\to \mathscr{O}_{Y'}(B)\to \mathscr{O}_{Y'}(B')\to \mathscr{O}_{B_1}\to 0.
\]
Since $\pi'$ is finite, $\pi'$ is an exact functor. Thus, there is an induced exact sequence
\[
0\to \pi'_*\mathscr{O}_{Y'}(B)\to \pi'_*\mathscr{O}_{Y'}(B')\to \pi_*'\mathscr{O}_{B_1}\to 0.
\]
Note that $\pi_*'\mathscr{O}_{B_1}=\pi_*'j_*j^*\mathscr{O}_{Y'}=j'_*\pi''_*\mathscr{O}_{B_1}$, where $j:B_1\to Y'$, $\pi''=\pi'|_{B_1}:B_1\to V$ and $j':V\to U$.
Since $B_1$ and $V$ are regular, $\pi''$ is finite flat of degree $r=[K(B_1):K(V)]$.
Then $\pi''_*\mathscr{O}_{B_1}$ is a vector bundle of rank $r$ on $V$. 
By \cite[Example 15.3.1]{fulton2013intersection}, we have $\det(j'_*\pi''_*\mathscr{O}_{B_1})\cong \mathscr{O}_U(rV)\cong \mathscr{O}_U(\pi'_*B_1)$.
Then 
\begin{align*}
\det(\pi'_*\mathscr{O}_{Y'}(B))&\cong \det(\pi'_*\mathscr{O}_{Y'}(B'))\otimes \det( \pi_*'\mathscr{O}_{B_1})^{-1}\\
&\cong \det(\pi'_*\mathscr{O}_{Y'})\otimes \mathscr{O}_U(\pi'_*(B'))\otimes \mathscr{O}_U(-\pi'_*B_1)\\
&\cong det(\pi'_*\mathscr{O}_{Y'})\otimes \mathscr{O}_U(\pi'_*(B)).
\end{align*}

\end{proof}

\subsection{Dualizing sheaf}

In this section, we collect some facts about dualizing sheaf for a finite morphism. 
We refer the reader to \cite{liu2006algebraic} for a more detailed discussion.


\begin{lemma}[Proposition 4.25, Lemma 4.26, Theorem 4.32 in \cite{liu2006algebraic}]\label{lemma: dualizing sheaf}
Let $f:X\to Y$ be a finite morphism of locally Noetherian schemes. Let $\mathscr{F}$ (resp. $\mathscr{G}$) be a quasi-coherent sheaf on $X$ (resp. on $Y$ ). Let  $f^!\mathscr{G}=\mathscr{H}om_{\mathscr{O}_Y}(f_*\mathscr{O}_X,\mathscr{G})$. 
\begin{enumerate}
    \item The sheaf $f^!\mathscr{G}$ is canonically endowed with the structure of a quasi-coherent $\mathscr{O}_X$-module and with a homomorphism $\mathrm{tr}_{\mathscr{G}}:f_*f^!\mathscr{G}\to \mathscr{G}$. 
    \item The canonical homomorphism $$f_*\mathscr{H}om_{\mathscr{O}_X}(\mathscr{F},f^!\mathscr{G})\to \mathscr{H}om_{\mathscr{O}_Y}(f_*\mathscr{F},\mathscr{G})$$ induced by $\mathrm{tr}_{\mathscr{G}}$ is an isomorphism. 
\item Let $\omega_f=f^!\mathscr{O}_Y$ and $\mathrm{tr}_f=\mathrm{tr}_{\mathscr{O}_Y}$. Then $(\omega_f, \mathrm{tr}_f)$ is the dualizing sheaf for $f$.
\item Let $f$ be projective and decompose into a finite morphism $\pi:X\to Z$ followed by a finite projective morphism $g:Z\to Y$. If $\pi$ is flat or if $\omega_g$ is locally free, then $\omega_f=\omega_{\pi}\otimes_{\mathscr{O}_X}\pi^*\omega_g$.
\item If $f$ is a flat projective local complete intersection, then the dualizing sheaf $\omega_f$ is a  invertible sheaf and is isomorphic to the canonical sheaf $\omega_{X/Y}$ of $f$. In particular, if $f$ is smooth, then $\omega_f=\mathscr{O}_X$. 
\end{enumerate}
\end{lemma}


\subsection{Hitchin morphism and spectral data morphism}

In this section, we fix a connected reductive group $G$ over $\mathbb{K}$ with Lie algebra $\mathfrak{g}$, a maximal torus $T$ of $G$ with Lie algebra $\mathfrak{t}$, the Weyl group $W=N_G(T)/T$, and discuss three moduli stacks on $X$: the moduli stack of Higgs bundles $\mathscr{M}_{X,G}$; the moduli stack of the Hitchin base $\mathscr{A}_{X,G}$; the moduli stack of spectral data $\mathscr{B}_{X,G}$, and Hitchin morphism $h_{X,G}$ and spectral data morphism $sd_{X,G}$. We refer the reader to \cite{chen2020hitchin} and \cite{song2023image} for a more detailed discussion.

Let $\mathfrak{C}^2_{\mathfrak{g}}\subseteq \mathfrak{g}^2$ be the commuting scheme, which is defined as the scheme-theoretic fiber of the commutator map over the zero
\[\mathfrak{g}^2\rightarrow \mathfrak{g}, \ \ (x_1, x_2)\mapsto [x_1, x_2].\] 
There are two group actions on $\mathfrak{C}^2_{\mathfrak{g}}$:
\begin{enumerate}
\item the diagonal adjoint $G$-action on $\mathfrak{g}^2$ restricts to one on $\mathfrak{C}^2_{\mathfrak{g}}$.\\
\item For $h\in GL_2$, the $GL_2$-action on $\mathfrak{C}^2_{\mathfrak{g}}$ is defined as
$(x_1,x_2)\mapsto (x_1,x_2)h$.
\end{enumerate}

Let $\mathscr{M}=[\mathfrak{C}^2_{\mathfrak{g}}/G\times GL_2]$ be the quotient stack. Let $X\to [*/GL_2]=\mathbb{B}GL_2$ be the morphism corresponding to $T_X^*$. There is a natural morphism $\mathscr{M}\to \mathbb{B}GL_2$.
We consider the fiber product
\[\begin{tikzcd}
	{\mathscr{M}\times_{\mathbb{B}GL_2}X} & {\mathscr{M}} \\
	{X} & {\mathbb{B}GL_2}
	\arrow[from=1-1, to=1-2]
	\arrow[from=1-1, to=2-1]
	\arrow["\square"{description}, draw=none, from=1-1, to=2-2]
	\arrow[ from=1-2, to=2-2]
	\arrow[from=2-1, to=2-2].
\end{tikzcd}\]
Let $\mathscr{M}_{X,G}=Sect(X,\mathscr{M}\times_{\mathbb{B}GL_2}X)$ be the stack of section of $\mathscr{M}\times_{\mathbb{B}GL_2}X$ over $X$, i.e., for each $\mathbb{K}$-scheme $S$ the groupoid $\mathscr{M}_{X,G}(S)$ consists of maps over $X$:
\[
X\times S\to  \mathscr{M}\times_{\mathbb{B}GL_2}X.
\]
Equivalently, $\mathscr{M}_{X,G}(S)$ consists of a pair $(E,\theta)$ (called a $G$-Higgs bundle), where $E$ is a $G$-principal bundle over $X\times S$ and Higgs field $\theta$ is an element in 
$ H^0(X\times S,\mathrm{ad}(E)\otimes_{\mathscr{O}_{X\times S}} \Omega_{X\times S}^1)$,
where $\mathrm{ad}(E)$ is the adjoint vector bundle of $E$ satisfying the integrability condition $\theta \wedge\theta=0$. 
$\mathscr{M}_{X,G}$ is actually the stack of $G$-Higgs bundles on $X$.

Let $\mathfrak{c}=\mathfrak{g}/\!\!/ G$. Note that $\mathfrak{c}\cong \mathfrak{t}/\!\!/W$ and it  has homogeneous coordinates $c_1,\dots,c_n$ of degree $e_1,\dots,e_n$ where $\dim(\mathfrak{t})=n$. Take $A=\prod_{i=1}^nS^{e_i}\mathbb{A}^2$.
Since the morphism $\mathfrak{g}\to\mathfrak{c}$ is $G$-invariant and $\mathbb{G}_m$-equivariant, it induces a $G$-invariant morphism $\mathrm{pol}: \mathfrak{g}^2\to A$,
which embodies Weyl’s polarization method for the diagonal action of $G$ on $\mathfrak{g}^2$.
For $v=(v_1,v_2)\in \mathfrak{g}^2$, we have $\mathrm{pol}(v)=(f_1(v),\dots,f_n(v))$ such that $c_i(t_1v_1+t_2v_2)=\sum_{i_1+i_2=i}f_{i_1i_2}(v)t_1^{i_1}t_2^{i_2}$ and $f_i=(f_{i_1i_2}\mid i_1+i_2=i)\in K[S^i\mathbb{A}^2]$. 
We obtain a $G$-invariant morphism $h:\mathfrak{C}^2_{\mathfrak{g}}\to A$
to be called the universal Hitchin morphism in \cite{chen2020hitchin}.

We have a natural morphism $[A/GL_2]\to\mathbb{B}GL_2$ of stacks.
Let $\mathscr{A}_{X,G}=Sect(X,[A/GL_2]\times_{\mathbb{B}GL_2}X)$ be the stack of section of $[A/GL_2]\times_{\mathbb{B}GL_2}X$ over $X$, i.e., for each $\mathbb{K}$-scheme $S$ the groupoid $\mathscr{A}_{X,G}(S)$ consists of maps over $X$:
\[
X\times S\to [A/GL_2]\times_{\mathbb{B}GL_2}X.
\]
The moduli stack $\mathscr{A}_{X,G}$ is represented by $\oplus_{i=1}^nH^0(X, S^{e_i}T_X^*)$. We also use $\mathscr{A}_{X,G}$ for the scheme $\oplus_{i=1}^nH^0(X, S^{e_i}T_X^*)$, which is known as the Hitchin base. 

We have the commutative diagram 
\[\begin{tikzcd}
	{[\mathfrak{C}^2_{\mathfrak{g}}/(G\times GL_2)]} & {[A/GL_2]} \\
	& {\mathbb{B}GL_2.}
	\arrow["{[h]}", from=1-1, to=1-2]
	\arrow[from=1-1, to=2-2]
	\arrow[from=1-2, to=2-2]
\end{tikzcd}\]
Hence we obtain the Hitchin morphism
\[
h_{X,G}:\mathscr{M}_{X,G}\to \mathscr{A}_{X,G}.
\]

By \cite{li2024functions}, $\mathfrak{g}^2/\!\!/ G\cong \mathfrak{t}^2/\!\!/ W$.
The morphism $\mathfrak{t}\to \mathfrak{c}=\mathfrak{t}/\!\!/W$ is $W$-invariant and $\mathbb{G}_m$-equivariant.
Using Weyl’s polarization construction for the diagonal action of $W$ on $\mathfrak{t}^2$, we have a $W$-invariant morphism $pol_W:\mathfrak{t}^2/\!\!/W\to A$.
By \cite[Theorem 4.1, Remark 4.1]{chen2020hitchin}, the morphism $pol_W$ is finite, and there exists a unique reduced closed subscheme $B$ of $A$ such that $pol_W$ factors through a morphism $\mathfrak{t}^2/\!\!/W\to B$ which is a universal homeomorphism and normalization. For $G=GL_n,SL_n$ or types B or C,  $pol_W$ is a closed embedding and $\mathfrak{t}^2/\!\!/W\to B$ is an isomorphism.

We have a natural morphism $[B/GL_2]\to\mathbb{B}GL_2$ of stacks.
Let $\mathscr{B}_{X,G}=Sect(X,[B/GL_2]\times_{\mathbb{B}GL_2}X)$ be the stack of section of $[B/GL_2]\times_{\mathbb{B}GL_2}X$ over $X$.
The moduli stack $\mathscr{B}_{X,G}$ is represented by scheme $Sect(X,Fr(T_X^*)\times_{GL_2}B)$. We also use $\mathscr{B}_{X,G}$ for the scheme $Sect(X,Fr(T_X^*)\times_{GL_2}B)$. 
Note that if $G=GL_n$ then $\mathscr{B}_{X,G}=Sect(X, Sym^{n}(T_X^*/X))$.
An element $b\in\mathscr{B}_{X,G}$ is called a spectral datum and $\mathscr{B}_{X,G}$ is called the space of spectral data.

We have the commutative diagram 
\[\begin{tikzcd}
	{[\mathfrak{C}^2_{\mathfrak{g}}/(G\times GL_2)]} & {[A/GL_2]} \\
	{[B/GL_2]} & {\mathbb{B}GL_2}.
	\arrow[from=1-1, to=1-2]
	\arrow[from=1-1, to=2-1]
	\arrow[ from=1-2, to=2-2]
	\arrow[from=2-1, to=2-2]
\end{tikzcd}\]
Hence we obtain the spectral data morphism $sd_{X,G}:\mathscr{M}_{X,G}\to \mathscr{B}_{X,G}$ and the commutative diagram 
\[\begin{tikzcd}
	{\mathscr{M}_{X,G}} & {\mathscr{A}_{X,G}} \\
	& {\mathscr{B}_{X,G}}.
	\arrow["{h_{X,G}}", from=1-1, to=1-2]
	\arrow["{sd_{X,G}}"', from=1-1, to=2-2]
	\arrow[hook, from=2-2, to=1-2]
\end{tikzcd}\]

\subsection{The Hitchin fiber of $h_{X,GLn}$}

In this section, we collect some results
about the Hitchin fiber of $h_{X,GLn}$ and spectral cover, and we refer the reader to \cite{chen2020hitchin} and \cite{song2023image} for a more detailed discussion.

Take a $b\in \mathscr{B}_{X,GL_n}(\mathbb{K})$ such that $b(\eta)\in \{[\alpha_1,\dots,\alpha_n]\in Sym^n(\mathbb{A}^2_{\bar{k(\eta)}})\mid  \alpha_i\neq \alpha_j,i\neq j\}$, where $\eta$ is the generic point of $X$.
Let $\sigma_0=1$ and
\begin{align*}
    &\prod_{i=1}^n(t_1x_1+t_2x_2-t_1x_{i1}-t_2x_{i2})\\
&=\sum_{i=0}^n(-1)^i\sigma_i(t_1x_{11}+t_2x_{12},\dots,t_1x_{n1}+t_2x_{n2})(t_1x_1+t_2x_2)^{n-i}\\
&=\sum_{i_1+i_2=n}f_{i_1i_2}t_1^{i_1}t_2^{i_2}
\end{align*}
in $\mathbb{K}[x_{ij},x_j,t_j\mid 1\leq i\leq n,j=1,2]$.
Then $f_{i_1i_2}\in \mathbb{K}[x_{ij},x_1,x_2\mid 1\leq i\leq n,j=1,2]^{\mathfrak{S}_n}$.
Hence the ideal generated by $\{f_{i_1i_2}\mid i_1+i_2=n\}$ defines the closed subscheme $Cayley^n(\mathbb{A}^2)$ of $Sym^n(\mathbb{A}^2)\times_{\mathbb{K}}\mathbb{A}^2$.
Since this ideal is a $GL_2$-module, we can obtain the closed subscheme $Cayley^{n}(T_X^*/X)=Fr(T_X^*)\times_{GL_2} Cayley^n(\mathbb{A}^2)$ of $Sym^n(T_X^*/X)\times_XT_X^*$.

\begin{lemma}[Proposition 6.1, 6.2, 6.3 in \cite{chen2020hitchin}]\label{lem: X_b of ChenNgo}\ 

\begin{enumerate}
\item The projection $p':Cayley^n(\mathbb{A}^2)\to Sym^n(\mathbb{A}^2)$ is a finite morphism which
is \'etale over the open subset $Sym^n_{(1^n)}(\mathbb{A}^2)$ of $Sym^n(\mathbb{A}^2)$ such that $Sym^n_{(1^n)}(\mathbb{A}^2)(\mathbb{K})=\{[\alpha_1,\dots,\alpha_n]\in  Sym^n(\mathbb{A}^2)(\mathbb{K})\mid \forall i\neq j,\alpha_i\neq \alpha_j\}$.
    \item There is a universal homeomorphism $\mathbb{A}^2/\!\!/\mathfrak{S}_{n-1}\to Cayley^n(\mathbb{A}^2)$.
    \item There is a cartesian diagram
\[\begin{tikzcd}
	{X_{b}} & {Cayley^n(T^*_X/X)} \\
	{X} & {Sym^n(T^*_X/X)}.
	\arrow[from=1-1, to=1-2]
	\arrow["{p_b}", from=1-1, to=2-1]
	\arrow["\square"{description}, draw=none, from=1-1, to=2-2]
	\arrow["{p'}", from=1-2, to=2-2]
	\arrow["{b}", from=2-1, to=2-2]
\end{tikzcd}\]
    \item The fiber $h_{X,GL_n}^{-1}(b)$ is isomorphic to the stack of maximal Cohen--Macaulay sheaves of generic rank $1$ on the spectral cover $X_{b}$.
Moreover, if $\mathscr{F}$ is a maximal Cohen--Macaulay sheaf of generic rank $1$ over $X_{b}$, then $\mathscr{E}={p_b}_*\mathscr{F}$ is a vector bundle of rank $n$ over $X$, and it is naturally equipped with a Higgs field $\theta:\mathscr{T}_X\to {p_b}_*\mathscr{O}_{X_{b}}\to \mathscr{E}nd_{\mathscr{O}_X}(\mathscr{E})$.
\end{enumerate}
\end{lemma}

\begin{lemma}\label{lem: X_b eqdim}
The surface $X_{b}$ is equidimensional.
\end{lemma}
\begin{proof}
By Lemma \ref{lem: X_b of ChenNgo}, $Sym^n(T^*_X/X)$ is normal, $Cayley^n(T^*_X/X)$ is irreducible, $p'$ is finite and surjective. By \cite[Theorem 14.131]{gortz2020algebraic}, $p'$ is universally open, then $p_b$ is open. For any irreducible component $Z$ of $X_{b}$, take an open subset $\emptyset\neq U\subseteq Z $ of $X_{b}$, then $p_b(U)$ is open in $X$, then $\dim(U)=2$ since $p_b$ is finite, then $\dim(Z)=2$.
\end{proof}

Let $\mu$ be a partition of $n$. We will write $\mu=(\mu_1^{\alpha_1},\dots,\mu_s^{\alpha_s})$, where $0<\mu_1<\cdots<\mu_s$ and $\sum_{i=1}^s\mu_i\alpha_i=n$. Define a locally closed subset $Sym^n_{\mu}(T_X^*/X)$ of $Sym^n(T_X^*/X)$, whose a $K$-valued point is $[z_1,\dots,z_n]\in Sym^n(T_X^*/X)(K)$ such that there are $\beta_i\neq \beta_{j},1\leq i\neq j\leq  \alpha=\sum_{i=1}^s\alpha_i$, $\{z_1,\dots,z_n\}=\{\beta_1,\dots,\beta_{\alpha}\}$ and
\[
(\#\{j\mid z_j=\beta_i\})_{1\leq i\leq \alpha}=(\underbrace{\mu_1,\dots,\mu_{1}}_{\alpha_1},\dots,\underbrace{\mu_s,\dots,\mu_{s}}_{\alpha_s}).
\]
Then $Sym^n_{(1^n)}(T_X^*/X)(\mathbb{K})=\{[\beta_1,\dots,\beta_n]\in  Sym^n(T_X^*/X)(\mathbb{K})\mid \forall i\neq j,\beta_i\neq \beta_j\}$ and there is a stratification of $Sym^n(T_X^*/X)$:
\[
Sym^n(T_X^*/X)=\bigsqcup_{\mu \text{ ranges over all partitions of } n}Sym^n_{\mu}(T_X^*/X).
\]

Given a partition $\mu=(\mu_1^{\alpha_1},\dots,\mu_s^{\alpha_s})$ of $n$, we define a morphism 
\[
\tau_{\mu}':(\underbrace{T_X^*\times_X\cdots T_X^*}_{\alpha_1})\times_X\cdots\times_X (\underbrace{T_X^*\times_X\cdots T_X^*}_{\alpha_s})\xrightarrow{(\overbrace{Id,\dots,Id}^{\mu_1})\times \cdots \times (\overbrace{Id,\dots,Id}^{\mu_s})} 
\]
\[
\underbrace{T_X^*\times_X\cdots T_X^*}_{n}\to Sym^n(T_X^*/X).
\]
Since $\tau_{\mu}'$ is $\mathfrak{S}_{\alpha_1}\times\cdots\times \mathfrak{S}_{\alpha_s}$-invariant, we have the morphism 
\[
\tau_{\mu}:Sym^{\alpha_1}(T_X^*/X)\times_X\cdots\times_X Sym^{\alpha_s}(T_X^*/X)\to Sym^n(T_X^*/X).
\]

\begin{lemma}[Proposition 3.3, Theorem 3.8 and Lemma 4.2 in \cite{song2023image}]\label{lem: decomposition in SS}\ 

\begin{enumerate}
    \item Given a partition $\mu=(\mu_1^{\alpha_1},\dots,\mu_s^{\alpha_s})$, the morphism $\tau_{\mu}$ gives the normalization $\overline{
Sym^n_{\mu}(T_X^*/X)}$. 
\item Given a spectral datum $b: X\to Sym^n(T_X^*/X)$, suppose the generic point $\eta$ of $X$ is mapped to the stratum $Sym^n_{\mu}(T_X^*/X)$ with $\mu=(\mu_1^{\alpha_1},\dots,\mu_s^{\alpha_s})$. Then
there exist spectral data $$b_i: X\to Sym^{\alpha_i}(T_X^*/X),1\leq i\leq s,$$ such that $b_i(\eta)\in Sym_{(1^{\alpha_i})}^{\alpha_i}(T_X^*/X)$ for $1\leq i\leq s$ and $b$ is
the following composition of morphisms
\[
X\xrightarrow{(b_1,\dots,b_s)} Sym^{\alpha_1}(T_X^*/X)\times_X\cdots\times_X Sym^{\alpha_s}(T_X^*/X)\xrightarrow{\tau_{\mu}} Sym^n(T_X^*/X).
\]
\item Given a pair of spectral data $b_i:X \to Sym^{n_i}(T_X^*/X),i=1,2$, let $(E_i, \phi_i)$ be Higgs bundles on $X$ whose spectral data are $b_i$. Then $(E_1\oplus E_2, \phi_1\oplus \phi_2)$ is a Higgs bundle with the spectral datum $(b_1,b_2): X\to Sym^{n_1}(T_X^*/X)\times_X Sym^{n_2}(T_X^*/X)\to Sym^{n_1+n_2}(T^*_X/X)$.

\end{enumerate}
\end{lemma}

\section{The case: $SL_n$ for $2\nmid n$}

In this section, we fix $2\nmid n$ and the usual faithful representation $SL_n\to GL_n$.
Note that the category of $GL_n$-bundles on $X$ is equivalent to the category of vector bundles of rank $n$ on $X$. Then the category of $SL_n$-bundles on $X$ is equivalent to the category of vector bundles of rank $n$ with trivial determinant on $X$.
Then a $SL_n$-Higgs bundle is a $GL_n$-Higgs bundle $(E,\theta)$ over $X$ such that $\det(E)\cong \mathscr{O}_X$ and $tr(\theta)=0$. 

Let $Sym^{n,tr=0}(\mathbb{A}^2)=\mathrm{Spec}\mathbb{K}[\sigma_{i_1i_2}\mid 1\leq i_1+i_2\leq n]/(\sigma_{10},\sigma_{01})$ be the closed subscheme of $Sym^n(\mathbb{A}^2)$, where $\sigma_{i_1i_2}$ is defined in \Cref{subsection: notations}.
Let $Sym^{n,tr=0}(T^*_X/X)= Fr(T^*_X)\times_{GL_2}Sym^{n,tr=0}(\mathbb{A}^2)$ be the closed subscheme of $Sym^n(T^*_X/X)$.
Then
\[\mathscr{B}_{X,SL_n}=Sect(X, Sym^{n,tr=0}(T_X^*/X))\]
is a closed subscheme of $\mathscr{B}_{X,GL_n}=Sect(X, Sym^{n}(T_X^*/X))$. 
The Hitchin morphism $h_{X,SL_n}:\mathscr{M}_{X,SL_n}\to \mathscr{A}_{X,SL_n}$ factors through $sd_{X,SL_n}:\mathscr{M}_{X,SL_n}\to \mathscr{B}_{X,SL_n}$.
In this section, we will prove that $sd_{X,SL_n}$ is surjective.

First, we fix $m\geq 1$ and $b\in \mathscr{B}_{X,GL_m}(\mathbb{K})$ such that $b(\eta)\in Sym_{(1^m)}^m(T_X^*/X)$, where $\eta$ is the generic point of $X$.

By Lemma \ref{lem: X_b of ChenNgo}, there is a cartesian diagram
\[\begin{tikzcd}
	{X_{b}} & {Cayley^n(T^*_X/X)} \\
	{X} & {Sym^n(T^*_X/X)}.
	\arrow[from=1-1, to=1-2]
	\arrow["{p_b}", from=1-1, to=2-1]
	\arrow["\square"{description}, draw=none, from=1-1, to=2-2]
	\arrow["{p'}", from=1-2, to=2-2]
	\arrow["{b}", from=2-1, to=2-2]
\end{tikzcd}\]
Let $X_1,\dots,X_r$ be the irreducible components of $X_{b}$. 
For any $1\leq i\leq r$, let $\pi_i':Y_i\to X_i$ be the normalization morphism, $\pi_i=p_b\circ \pi_i':Y_i\to X$ and $\pi_{b}:Y_{b}=\bigsqcup_{1\leq i\leq r} Y_i\to X_{b}\to X$ be the natural morphisms. 

\begin{lemma}\label{lemma: pi flat}
The morphism $\pi_b$ is finite flat.
\end{lemma}
\begin{proof}
By Lemma \ref{lem: X_b eqdim}, $\dim(Y_i)=2$. Note that $\pi_i$ is finite. Since the surface $Y_i$ is normal and $X$ is smooth, $\pi_i$ is flat. So $\pi_b$ is finite flat.
\end{proof}

\begin{lemma}\label{lem: det^2(pi_b*_Y)=L}
The $\mathscr{O}_X$-module ${\pi_{b}}_*\mathscr{O}_{Y_{b}}$ is locally free and $\det({\pi_{b}}_*\mathscr{O}_{Y_{b}})^2\cong \mathscr{O}_X({-\pi_{b}}_*R_{b})$ for some cycle $R_{b}\in Z^1(Y_{b})$.
\end{lemma}
\begin{proof}
By Lemma \ref{lemma: pi flat}, $\pi_b$ is finite flat, then ${\pi_{b}}_*\mathscr{O}_{Y_{b}}$ is locally free.
Let $U=X\setminus\pi_b(Y_b^{sing})$, $W=\pi_b^{-1}(U)$, and $f=\pi_b|_{W}$. For any affine open subset $V$ of $U$, $f|_{{W}_V}:{W}_V=f^{-1}(V)\to V$ is finite, then $f|_{{W}_V}$ is a flat finite projective local complete intersection, then $({f}^!\mathscr{O}_U)|_{{W}_V}\cong {f|_{{W}_V}}^!\mathscr{O}_V$ is an invertible sheaf by Lemma \ref{lemma: dualizing sheaf}. Hence ${f}^!\mathscr{O}_U$ is an invertible sheaf. 
By Lemma \ref{lemma: dualizing sheaf}, $f_*f^!\mathscr{O}_U\cong (f_*\mathscr{O}_{W})^{\vee}$.
Since $W$ is open in $Y_b$ and $\dim(Y_b\setminus W)=0$, $Z^1(Y_b)\cong Z^1(W)$. 
Let $f^!\mathscr{O}_U=\mathscr{O}_{W}(R_b|_W)$ for some cycle $R_b\in Z^1(Y_b)$.
Since $f_*f^!\mathscr{O}_U\cong (f_*\mathscr{O}_{W})^{\vee}$, we have $({\pi_b}_*\mathscr{O}_{Y_b}(R_b))|_U\cong f_*\mathscr{O}_{W}(R_b|_W)=f_*f^!\mathscr{O}_U\cong (f_*\mathscr{O}_{W})^{\vee}\cong ({\pi_b}_*\mathscr{O}_{Y_b})^{\vee}|_U$.
Since $\dim(X\setminus U)=0$ and ${\pi_{b}}_*\mathscr{O}_{Y_{b}}$ is locally free, we have ${\pi_b}_*\mathscr{O}_{Y_b}(R_b)\cong ({\pi_b}_*\mathscr{O}_{Y_b})^{\vee}$, then $\det({\pi_b}_*\mathscr{O}_{Y_b}(R_b))\cong (\det({\pi_b}_*\mathscr{O}_{Y_b}))^{\vee}$.
By Lemma \ref{fomular}, we have $\det({\pi_b}_*\mathscr{O}_{Y_b}(R_b))\cong \det({\pi_b}_*\mathscr{O}_{Y_b})\otimes \mathscr{O}_X({\pi_b}_*R_b)$. Hence  $\det({\pi_{b}}_*\mathscr{O}_{Y_{b}})^2\cong \mathscr{O}_X({-\pi_{b}}_*R_{b})$.
\end{proof}

\begin{lemma}\label{lem: det(pi_b*_Y(sD))}
For any $s\in \mathbb{Z}$, we see that $({\pi_{b}}_*\mathscr{O}_{Y_{b}}(sR_{b}),\theta_{b})\in h_{X,GL_m}^{-1}(b)$ and  $\det({\pi_{b}}_*\mathscr{O}_{Y_{b}}(sR_{b}))\cong \det({\pi_{b}}_*\mathscr{O}_{Y_{b}})^{1-2s}$, where Higgs field $\theta_{b}:\mathscr{T}_X\to {p_b}_*\mathscr{O}_{X_{b}}\to \mathscr{E}nd_{\mathscr{O}_X}({\pi_{b}}_*\mathscr{O}_{Y_{b}}(sR_{b}))$.
\end{lemma}
\begin{proof}
For any $1\leq i\leq r$, $Y_i$ is normal and $\mathscr{O}_{Y_{b}}(sR_{b})|_{Y_i}$ is reflexive, then $\mathscr{O}_{Y_{b}}(sR_{b})|_{Y_i}$ is a Cohen--Macaulay $\mathscr{O}_{Y_i}$-module by Lemma \ref{lemma: reflexive to (S_2)}. Since the support of $\mathscr{O}_{Y_{b}}(sR_{b})|_{Y_i}$ is $Y_i$, $\mathscr{O}_{Y_{b}}(sR_{b})|_{Y_i}$ is maximal Cohen--Macaulay. So $\mathscr{O}_{Y_{b}}(sR_{b})$ is a maximal Cohen--Macaulay $\mathscr{O}_{Y_b}$-module of generic rank $1$.
Since the natural morphism $f:Y_b=\bigsqcup_{1\leq i\leq r} Y_i\to \bigsqcup_{1\leq i\leq r} X_i\to X_b$ is finite surjective, we see that $f_*\mathscr{O}_{Y_{b}}(sR_{b})$ is a maximal Cohen--Macaulay $\mathscr{O}_{X_b}$-module of generic rank $1$ by Lemma \ref{lemma: A to B, E CM}.
Since $\pi_b=p_b\circ f$, we have ${\pi_{b}}_*\mathscr{O}_{Y_{b}}(sR_{b})\cong {p_{b}}_*(f_*\mathscr{O}_{Y_{b}}(sR_{b}))$, then $({\pi_{b}}_*\mathscr{O}_{Y_{b}}(sR_{b}),\theta_{b})\in h_{X,GL_m}^{-1}(b)$ by Lemma \ref{lem: X_b of ChenNgo}.
By Lemma \ref{fomular} and Lemma \ref{lem: det^2(pi_b*_Y)=L}, we have
$\det({\pi_{b}}_*\mathscr{O}_{Y_{b}}(sR_{b}))\cong \det({\pi_b}_*\mathscr{O}_{Y_b})\otimes \mathscr{O}_X({\pi_b}_*sR_b)\cong \det({\pi_{b}}_*\mathscr{O}_{Y_{b}})^{1-2s}$.
\end{proof}

\begin{theorem}\label{thm: sd_SL(2n+1) surj}
     The morphism $sd_{X,SL_n}$ is surjective.
\end{theorem}
\begin{proof}
For any $b\in \mathscr{B}_{X,SL_n}(\mathbb{K})$, then there is a partition $\mu=(\mu_1^{\alpha_1},\dots,\mu_s^{\alpha_s})$ of $n$ such that 
$0<\mu_1<\cdots<\mu_s,0<\alpha_i,\sum_{i=1}^s\mu_i\alpha_i=n$ and $b(\eta)\in Sym^n_{\mu}(T_X^*/X)$. 
By Lemma \ref{lem: decomposition in SS}, there exist spectral data $$b_i: X\to Sym^{\alpha_i}(T_X^*/X),1\leq i\leq s,$$
such that $b_i(\eta)\in Sym_{(1^{\alpha_i})}^{\alpha_i}(T_X^*/X)$ for $1\leq i\leq s$ and $b$ is
the following composition of morphisms
\[
X\xrightarrow{(b_1,\dots,b_s)} Sym^{\alpha_1}(T_X^*/X)\times_X\cdots\times_X Sym^{\alpha_s}(T_X^*/X)\xrightarrow{\tau_{\mu}} Sym^n(T_X^*/X).
\]

Let $\Lambda_1=\{1\leq i\leq s\mid 2\mid \mu_i\},\Lambda_2=\{1\leq i\leq s\mid 2\nmid \mu_i\alpha_i\}$ and $\Lambda_3=\{1\leq i\leq s\mid 2\nmid \mu_i,2\mid \alpha_i\}$.
For any $i\in \Lambda_2$, there is $(\mathscr{E}_i,\theta_{b_i})\in h_{X,GL_{\alpha_i}}^{-1}(b_i)$ where $\mathscr{E}_i={\pi_{b_i}}_*\left(\mathscr{O}_{Y_{b_i}}(\frac{\alpha_i+1}{2}R_{b_i})\right)$ and $\det(\mathscr{E}_i)\cong \det({\pi_{b_i}}_*\mathscr{O}_{Y_{b_i}})^{-\alpha_i}$ by Lemma \ref{lem: det(pi_b*_Y(sD))}.
Then $(\mathscr{E}_i',\theta_{b_i}')\in h_{X,GL_{\alpha_i}}^{-1}(b_i)$ and $\det(\mathscr{E}_i')\cong \mathscr{O}_{X}$, where $\mathscr{E}_i'=\mathscr{E}_i\otimes_{\mathscr{O}_X}\det({\pi_{b_i}}_*\mathscr{O}_{Y_{b_i}})$ and $\theta_{b_i}'$ is induced by $\theta_{b_i}$.
For any $i\in \Lambda_1\cup \Lambda_3$, there are $(\mathscr{E}_i,\theta_{b_i}),(\mathscr{E}_i',\theta_{b_i}')\in h_{X,GL_{\alpha_i}}^{-1}(b_i)$ where $\mathscr{E}_i={\pi_{b_i}}_*\mathscr{O}_{Y_{b_i}}(R_{b_i}),\mathscr{E}_i'={\pi_{b_i}}_*\mathscr{O}_{Y_{b_i}}(-R_{b_i})$ and  $\det(\mathscr{E}_i')\cong \det({\pi_{b_i}}_*\mathscr{O}_{Y_{b_i}})\cong \det(\mathscr{E}_i)^{-1}$ by Lemma \ref{lem: det(pi_b*_Y(sD))}.

If $\Lambda_3=\emptyset$, let $\mathscr{E}=(\oplus_{i\in \Lambda_1}(\mathscr{E}_i^{\oplus \frac{\mu_i}{2}}\oplus {\mathscr{E}_i'}^{\oplus \frac{\mu_i}{2}}))\oplus (\oplus_{j\in \Lambda_2}{\mathscr{E}_j'}^{\oplus \mu_j})$ and $\theta=(\oplus_{i\in \Lambda_1}(\theta_{b_i}^{\oplus \frac{\mu_i}{2}}\oplus {\theta_{b_i}'}^{\oplus \frac{\mu_i}{2}}))\oplus (\oplus_{j\in \Lambda_2}{\theta_{b_j}'}^{\oplus \mu_j})$, then $(\mathscr{E},\theta)\in h_{X,GL_n}^{-1}(b)$ by Lemma \ref{lem: decomposition in SS} and $\det(\mathscr{E})\cong \mathscr{O}_X$, then $(\mathscr{E},\theta)\in h_{X,SL_n}^{-1}(b)$.
If $\Lambda_3\neq \emptyset$, then $\Lambda_2\neq \emptyset$ since $2\nmid n$. Take $k\in \Lambda_2,k'\in \Lambda_3$. Then $2\nmid m$ where $m=\alpha_k+\alpha_{k'}$ and there is $b'\in \mathscr{B}_{X,GL_m}(\mathbb{K})$ such that
\[
b':X\xrightarrow{(b_k,b_{k'})} Sym^{\alpha_k}(T_X^*/X)\times_X Sym^{\alpha_{k'}}(T_X^*/X)\to Sym^m(T_X^*/X).
\]
Thus, there is $(\mathscr{E}_{b'},\theta_{b'})\in h_{X,GL_{m}}^{-1}(b')$ where $\mathscr{E}_{b'}={\pi_{b'}}_*\left(\mathscr{O}_{Y_{b'}}(\frac{m+1}{2}R_{b'})\right)$ and  $\det(\mathscr{E}_{b'})\cong \det({\pi_{b'}}_*\mathscr{O}_{Y_{b'}})^{-m}$ by Lemma \ref{lem: det(pi_b*_Y(sD))}.
Then $(\mathscr{E}_{b'}',\theta_{b'}')\in h_{X,GL_{m}}^{-1}(b')$ and $\det(\mathscr{E}_{b'}')\cong \mathscr{O}_{X}$, where $\mathscr{E}_{b'}'=\mathscr{E}_{b'}\otimes_{\mathscr{O}_X}\det({\pi_{b'}}_*\mathscr{O}_{Y_{b'}})$ and $\theta_{b'}'$ is induced by $\theta_{b'}$. 
Let 
\begin{align*}
    \mathscr{E}'=&\left(\oplus_{i\in \Lambda_1\cup \Lambda_3\setminus\{k'\}}(\mathscr{E}_i^{\oplus \frac{\mu_i}{2}}\oplus {\mathscr{E}_i'}^{\oplus \frac{\mu_i}{2}})\right)\oplus \left(\oplus_{j\in \Lambda_2\setminus\{k\}}{\mathscr{E}_j'}^{\oplus \mu_j}\right)\\
   & \oplus \left(\oplus_{l\in \mathbb{Z},0<2l\leq \mu_{k'}-1 }(\mathscr{E}_{k'}\oplus {\mathscr{E}_{k'}'})\right)\oplus\left(\oplus_{l\in \mathbb{Z},0<l\leq \mu_{k}-1 }\mathscr{E}_k'\right)\oplus\mathscr{E}_{b'}',\\
   \theta'=&\left(\oplus_{i\in \Lambda_1\cup \Lambda_3\setminus\{k'\}}(\theta_{b_i}^{\oplus \frac{\mu_i}{2}}\oplus {\theta_{b_i}'}^{\oplus \frac{\mu_i}{2}})\right)\oplus \left(\oplus_{j\in \Lambda_2\setminus\{k\}}\theta_{b_j}^{\oplus \mu_j}\right)\\
   &\oplus \left(\oplus_{l\in \mathbb{Z},0<2l\leq \mu_{k'}-1 }(\theta_{k'}\oplus {\theta_{k'}'})\right)\oplus\left(\oplus_{l\in \mathbb{Z},0<l\leq \mu_{k}-1 }\theta_k\right)\oplus\theta_{b'}'.
\end{align*}
Then $(\mathscr{E}',\theta')\in h_{X,GL_n}^{-1}(b)$ by Lemma \ref{lem: decomposition in SS} and $\det(\mathscr{E}')\cong \mathscr{O}_X$, then $(\mathscr{E}',\theta')\in h_{X,SL_n}^{-1}(b)$.
\end{proof}

\section{The case: $X=C_1\times_{\mathbb{K}} C_2$}

In this section, we fix $X=C_1\times_{\mathbb{K}} C_2$ where $C_i$ is a smooth projective irreducible curve over $\mathbb{K}$ for $i=1,2$, and we will prove that $sd_{X,G}$ is surjective for $G=SL_n,Sp_{2n}$.

For $i=1,2$, let $\{U_{ij}=\mathrm{Spec}A_{ij}\}_{j\in \Lambda_i}$ be an affine open cover of $C_i$ such that  $\mathscr{T}_{C_i}|_{U_{ij}}\cong \widetilde{A_{ij}x_{ij}}$.
Let $\mathcal{U}_i=U_{1i'}\times_{\mathbb{K}} U_{2i''}$ for any $i=(i',i'')\in \Lambda=\Lambda_1\times \Lambda_2$.
Then $X=\cup_{i\in \Lambda}\mathcal{U}_i$ and $T_X^*|_{\mathcal{U}_i}\cong \mathrm{Spec}\ A_{i}[x_{1i},x_{2i}]$ for any $i=(i',i'')\in \Lambda$ where $A_{i}=A_{1i'}\otimes_{\mathbb{K}} A_{2i''}$.

\subsection{The case: $SL_{n}$}

Let $b\in \mathscr{B}_{X,GL_n}(\mathbb{K})$ such that $b(\eta)\in \{[\alpha_1,\dots,\alpha_n]\in Sym^m(\mathbb{A}^2_{\bar{k(\eta)}})\mid  \alpha_i\neq \alpha_j,i\neq j\}$, where $\eta$ is the generic point of $X$. Let $X_{b}$ be the spectral cover of $X$ with spectral data $b$. Then there is a cartesian diagram 
\[\begin{tikzcd}
	{X_{b}} & {Cayley^n(T^*_X/X)} \\
	{X} & {Sym^n(T^*_X/X)}.
	\arrow[from=1-1, to=1-2]
	\arrow["p_b", from=1-1, to=2-1]
	\arrow["\square"{description}, draw=none, from=1-1, to=2-2]
	\arrow[from=1-2, to=2-2]
	\arrow["{b}", from=2-1, to=2-2]
\end{tikzcd}\]
Consider the natural map
\[
\oplus_{i=1}^nH^0(X,S^i\Omega_X)\to (\oplus_{i=1}^nH^0(C_1,S^i\Omega_{C_1}),\oplus_{i=1}^nH^0(C_2,S^i\Omega_{C_2}))
\]
and let $(b_1,b_2)$ be the image of $b$.
Let $C_{ib_i}$ be the spectral cover of $C_i$ with spectral data $b_i$ for $i=1,2$. Then $C_{ib_i}$ is a closed subscheme of $T_{C_i}^*$ and there is a cartesian diagram
\[\begin{tikzcd}
	{C_{ib_i}} & {Cayley^n(T^*_{C_i}/C_i)} \\
	{C_i} & {Sym^n(T^*_{C_i}/C_i)}.
	\arrow[from=1-1, to=1-2]
	\arrow["{p_i}", from=1-1, to=2-1]
	\arrow["\square"{description}, draw=none, from=1-1, to=2-2]
	\arrow["{p_i'}", from=1-2, to=2-2]
	\arrow["{b_i}", from=2-1, to=2-2]
\end{tikzcd}\]

\begin{lemma}\label{lem: X_b to C_1bC_2b}
    The closed immersion $X_b\to T_X^*\cong T_{C_1}^*\times_{\mathbb{K}}T_{C_2}^*$ factors through the closed immersion $ C_{1b_1}\times_{\mathbb{K}}C_{2b_2}\to T_{C_1}^*\times_{\mathbb{K}}T_{C_2}^*$. And the image of the closed immersion $X_b\to T_X^*\cong T_{C_1}^*\times_{\mathbb{K}}T_{C_2}^*$ is the union of some irreducible components of $C_{1b_1}\times_{\mathbb{K}} C_{2b_2}$.
\end{lemma}
\begin{proof}
For any $j=(j_1,j_2)\in \Lambda$ and $i=1,2$, $b_i|_{U_{ij_i}}$ induces a morphism $\phi_{ij_i}:A_{ij_i}[\sigma_1,\dots,\sigma_n]\to A_{ij_i}$, and $b|_{\mathcal{U}_j}$ induces a morphism $$\phi_{j}:A_{j}[\sigma_{s_1s_2}\mid s_1,s_2\geq 0,1\leq s_1+s_2\leq n]\to A_{j}=A_{1j_1}\otimes_{\mathbb{K}}A_{2j_2},$$
where $\phi_j(\sigma_{s0})=\phi_{1j_1}(\sigma_s)\otimes 1$ and $\phi_j(\sigma_{0s})=1\otimes \phi_{2j_2}(\sigma_s)$ for any $1\leq s\leq n$, and $\sigma_s,\sigma_{s_1s_2}$ are defined in \Cref{subsection: notations}.
Then
$$p_i^{-1}(U_{ij_i})=\mathrm{Spec}A_{ij_i}[x_i]/(f_{ij_i}),\ f_{ij_i}=x_i^n+\sum_{l=1}^n(-1)^l\phi_{ij_i}(\sigma_l)x_i^{n-l},$$
and $p_b^{-1}(\mathcal{U}_j)=\mathrm{Spec}A_{j}[x_1,x_2]/(g_{js}\mid 0\leq s\leq n)$, where $\sigma_{00}=1$ and
\[
g_{js}=\sum_{l=0}^n(-1)^l\sum_{\max\{s+l-n,0\}\leq k\leq \min\{s,l\}}\phi_j(\sigma_{k,l-k})\binom{n-l}{s-k}x_1^{s-k}x_2^{n-s-l+k}.
\]
Consider the natural morphism $A_{ij_i}[x_i]\to A_{j}[x_1,x_2]$ for $i=1,2$, the image of $f_{ij_i}$ is $g_{jq}$ where $q=n$ if $i=1$ and $q=0$ if $i=2$.
So the natural morphism $A_{1j_1}[x_1]/(f_{1j_1})\otimes_{\mathbb{K}}A_{2j_2}[x_2]/(f_{2j_2}) \to A_{j}[x_1,x_2]/(g_{js}\mid 0\leq s\leq n)$ is surjective.
Then the closed immersion $p_b^{-1}(\mathcal{U}_j)\to T_X^*|_{\mathcal{U}_j}$ factors through the closed immersion $p_b^{-1}(\mathcal{U}_j)\to p_1^{-1}(U_{1j_1})\times_{\mathbb{K}}p_2^{-1}(U_{2j_2})$.
Hence the closed immersion $X_b\to T_X^*$ factors through the closed immersion $X_b\to C_{1b_1}\times_{\mathbb{K}}C_{2b_2}$.
By Lemma \ref{lem: X_b eqdim}, the surface $X_b$ is equidimensional.
Note that the curves $C_{1b_1} $and $ C_{2b_2}$ are equidimensional. 
Thus, the image of the closed immersion $X_b\to T_X^*\cong T_{C_1}^*\times_{\mathbb{K}}T_{C_2}^*$ is the union of some irreducible components of $C_{1b_1}\times_{\mathbb{K}} C_{2b_2}$.
\end{proof}

Let $X_1,\dots,X_r$ be the irreducible components of $X_{b}$. 
Let $\pi_i:Y_i\to X_i$ be the normalization morphism for any $1\leq i\leq r$ and $\pi_{b}:Y_{b}=\bigsqcup_{1\leq i\leq r} Y_i\to X_{b}\to X$ be the natural morphism.

\begin{lemma}\label{lem: Y_i=C times C' sm}
For $1\leq i\leq r$, $Y_i\cong C'_i\times_{\mathbb{K}}C''_i$ for some smooth projective curves $C'_i$ and $C''_i$.
\end{lemma}
\begin{proof}
For $i=1,2$, let $Z_{i1},\dots,Z_{ir_i}$ be the irreducible components of $C_{ib_i}$. 
Let $Z_{ij}'\to Z_{ij}$ be the normalization morphism for any $i=1,2$ and $1\leq j\leq r_i$.
Since the curve $Z_{ij}$ is proper over $\mathbb{K}$, the curve $Z_{ij}'$ is smooth projective. By Lemma \ref{lem: X_b to C_1bC_2b}, $X_i\cong Z_{1j_i}\times_{\mathbb{K}}Z_{2s_i}$ for any $1\leq i\leq r$ and some $1\leq j_i\leq r_1$ and $1\leq s_i\leq r_2$. Then $Y_i\cong Z_{1j_i}'\times_{\mathbb{K}}Z_{2s_i}'$.
\end{proof}

\begin{lemma}\label{lem: (E,theta)}
There is a cycle $D\in Z^1(Y_b)$ and $\mathscr{E}_b={\pi_b}_*\mathscr{O}_{Y_b}(D)$ such that  $(\mathscr{E}_b,\theta_{b})\in h_{X,GL_n}^{-1}(b)$ and  $\det(\mathscr{E}_b)\cong \mathscr{O}_X$, where Higgs field $\theta_{b}:\mathscr{T}_X\to {p_b}_*\mathscr{O}_{X_{b}}\to \mathscr{E}nd_{\mathscr{O}_X}(\mathscr{E}_b)$.
\end{lemma}
\begin{proof}
By Lemma \ref{lem: Y_i=C times C' sm}, $Y_i\cong C'_i\times_{\mathbb{K}}C''_i$ for some projective smooth curves $C'_i$ and $C''_i$ over $\mathbb{K}$.
Note that $X=C_1\times_{\mathbb{K}}C_2$ is projective smooth over $\mathbb{K}$.
Then $\pi_b$ is a projective local complete intersection, and the canonical sheaf $\omega_{Y_i}\cong L_i^2$ and $\omega_{X}\cong L^2$ for some $L_i\in \mathrm{Pic}(Y_i)$ and $L\in \mathrm{Pic}(X)$.
By Lemma \ref{lemma: pi flat}, $\pi_b$ is flat, then $\pi_b$ is a flat projective local complete intersection
By Lemma \ref{lemma: dualizing sheaf}, the dualizing sheaf $\omega_{\pi_b}=\pi_b^!\mathscr{O}_X$ is isomorphic to the canonical sheaf $\omega_{X_b/Y}\cong \omega_{X_b}\otimes_{\mathscr{O}_{X_b}}\pi_b^*\omega_X^{-1}$.
Then there is a cycle $D'\in Z^1(Y_b)$ such that $\mathscr{O}_{Y_b}(2D')\cong \pi_b^!\mathscr{O}_X$.
By Lemma \ref{lemma: dualizing sheaf}, ${\pi_b}_*\mathscr{O}_{Y_b}(2D')\cong {\pi_b}_*\pi_b^!\mathscr{O}_X\cong ({\pi_b}_*\mathscr{O}_{Y_b})^{\vee}$, then $\det({\pi_b}_*\mathscr{O}_{Y_b}(2D'))\cong \det({\pi_b}_*\mathscr{O}_{Y_b})^{\vee}$.
By Lemma \ref{fomular}, we have $\det({\pi_b}_*\mathscr{O}_{Y_b}(2D'))\cong \det({\pi_b}_*\mathscr{O}_{Y_b})\otimes \mathscr{O}_X(2{\pi_b}_*D')$, then $(\det({\pi_b}_*\mathscr{O}_{Y_b}))^2\cong \mathscr{O}_X(-2{\pi_b}_*D')$.
Note that $NS(X)$ is free. Then $\det({\pi_b}_*\mathscr{O}_{Y_b})\cong \mathscr{O}_X(-{\pi_b}_*D')\otimes\mathscr{L}^{-n}$ for some $\mathscr{L}\in \mathrm{Pic}(X)$. 
Take a cycle $D\in Z^1(Y_b)$ such that $\mathscr{O}_{Y_b}(D)\cong \mathscr{O}_{Y_b}(D')\otimes {\pi_b}^*\mathscr{L}$. 
Let $\mathscr{E}_b={\pi_b}_*\mathscr{O}_{Y_b}(D)$.
Then $\det(\mathscr{E}_b)\cong \mathscr{O}_X$.

Since $Y_b$ is regular, we have $\mathscr{O}_{Y_b}(D)\in \mathrm{Pic}(Y_b)$, then $\mathscr{O}_{Y_b}(D)$ is a maximal Cohen--Macaulay $\mathscr{O}_{Y_b}$-module of rank $1$.
Since the natural morphism $f:Y_b=\bigsqcup_{1\leq i\leq r} Y_i\to \bigsqcup_{1\leq i\leq r} X_i\to X_b$ is finite surjective, we see that $f_*\mathscr{O}_{Y_b}(D)$ is a maximal Cohen--Macaulay $\mathscr{O}_{X_b}$-module of generic rank $1$ by Lemma \ref{lemma: A to B, E CM}.
Since $\pi_b=p_b\circ f$, we have $\mathscr{E}_b={\pi_{b}}_*\mathscr{O}_{Y_b}(D)\cong {p_{b}}_*(f_*\mathscr{O}_{Y_b}(D))$, then $(\mathscr{E}_b,\theta_{b})\in h_{X,GL_m}^{-1}(b)$ by Lemma \ref{lem: X_b of ChenNgo}.
\end{proof}

\begin{theorem}\label{thm: sd_SL(n) surj for X=C1 times C2}
     The morphism $sd_{X,SL_n}$ is surjective.
\end{theorem}
\begin{proof}
For any $b\in \mathscr{B}_{X,SL_n}(\mathbb{K})$, then there is a partition $\mu=(\mu_1^{\alpha_1},\dots,\mu_s^{\alpha_s})$ of $n$ such that 
$0<\mu_1<\cdots<\mu_s,0<\alpha_i,\sum_{i=1}^s\mu_i\alpha_i=n$ and $b(\eta)\in Sym^n_{\mu}(T_X^*/X)$. 
By Lemma \ref{lem: decomposition in SS}, there exist spectral data $$b_i: X\to Sym^{\alpha_i}(T_X^*/X),1\leq i\leq s,$$ such that $b_i(\eta)\in Sym_{(1^{\alpha_i})}^{\alpha_i}(T_X^*/X)$ for $1\leq i\leq s$ and $b$ is
the following composition of morphisms
\[
X\xrightarrow{(b_1,\dots,b_s)} Sym^{\alpha_1}(T_X^*/X)\times_X\cdots\times_X Sym^{\alpha_s}(T_X^*/X)\xrightarrow{\tau_{\mu}} Sym^n(T_X^*/X).
\]
By Lemma \ref{lem: (E,theta)}, there is $(\mathscr{E}_{b_i},\theta_{b_i})\in h_{X,GL_{\alpha_i}}^{-1}(b_i)$ and $\det(\mathscr{E}_{b_i})\cong \mathscr{O}_X$ for any $1\leq i\leq s$.
Let $\mathscr{E}=\oplus_{1\leq i\leq s}\mathscr{E}_{b_i}^{\oplus\mu_i}$ and $\theta=\oplus_{1\leq i\leq s}\theta_{b_i}^{\oplus\mu_i}$.
By Lemma \ref{lem: decomposition in SS}, $(\mathscr{E},\theta)\in h_{X,GL_n}^{-1}(b)$, then $(\mathscr{E},\theta)\in h_{X,SL_n}^{-1}(b)$ since $\det(\mathscr{E})\cong \mathscr{O}_X$.
\end{proof}

\subsection{The case: $Sp_{2n}$}

In this section, we fix the usual faithful representation $Sp_{2n}\to GL_{2n}$. A $Sp_{2n}$-Higgs bundle on $X$ is $(\mathscr{E},\theta,Q)$, where $(\mathscr{E},\theta)$ is Higgs bundle of rank $2n$ on $X$, and $Q:\mathscr{E}\otimes \mathscr{E}\to \mathscr{O}_X$ is a perfect alternating bilinear form such that $\theta$ is anti-self-adjoint with respect to $Q$. 

Note that $\mathfrak{g}/\!\!/G\cong \mathfrak{t}/\!\!/W\cong \mathrm{Spec}\mathbb{K}[E_1,\dots,E_{2n}]$ for $G=Sp_{2n}$ where $E_i$ is the  $i$-th symmetric function of $2n$ variables evaluated at the point $(x_1,\dots,x_n,-x_1,\dots,-x_n)$.
Since $\mathbb{K}[E_1,\dots,E_{2n}]\cong \mathbb{K}[\sigma_{2i}\mid 1\leq i\leq n]$, $B$ is the closed subscheme of $Sym^{2n}(\mathbb{A}^2)$.
Let $Sym^{2n,Sp}(\mathbb{A}^2)=\mathrm{Spec}\mathbb{K}[\sigma_{i_1i_2}\mid i_1+i_2=2i,1\leq i\leq n]$ be the closed subscheme of $Sym^{2n}(\mathbb{A}^2)$, where $\sigma_{i_1i_2}$ is defined in \Cref{subsection: notations}.
Let $Sym^{2n,Sp}(T^*_X/X)= Fr(T^*_X)\times_{GL_2}Sym^{2n,Sp}(\mathbb{A}^2)$ be the closed subscheme of $Sym^{2n}(T^*_X/X)$.
Then
\[\mathscr{B}_{X,Sp_{2n}}=Sect(X, Sym^{2n,Sp}(T_X^*/X))\]
is a closed subscheme of $\mathscr{B}_{X,GL_{2n}}=Sect(X, Sym^{2n}(T_X^*/X))$. 
The Hitchin morphism $h_{X,Sp_{2n}}:\mathscr{M}_{X,Sp_{2n}}\to \mathscr{A}_{X,Sp_{2n}}$ factors through $sd_{X,Sp_{2n}}:\mathscr{M}_{X,Sp_{2n}}\to \mathscr{B}_{X,Sp_{2n}}$.
In this section, we will prove that $sd_{X,Sp_{2n}}$ is surjective.

Let $b\in  \mathscr{B}_{X,GL_{2n}}(\mathbb{K})$ such that $b(\eta)\in \{[\alpha_1,\dots,\alpha_n,-\alpha_1,\dots,-\alpha_n]\in Sym^{2n}(\mathbb{A}_{\bar{k(\eta)}}^2)|\alpha_i\neq \alpha_j,i\neq j\}\subseteq Sym^{2n}_{(1^{2n})}(T_X^*/X)$, where $\eta$ is the generic point of $X$. Let $X_{b}$ be the spectral cover of $X$ with spectral data $b$. Then there is a cartesian diagram
\[\begin{tikzcd}
	{X_{b}} & {Cayley^{2n}(T^*_X/X)} \\
	{X} & {Sym^{2n}(T^*_X/X)}.
	\arrow[from=1-1, to=1-2]
	\arrow["p_b", from=1-1, to=2-1]
	\arrow["\square"{description}, draw=none, from=1-1, to=2-2]
	\arrow[from=1-2, to=2-2]
	\arrow["{b}", from=2-1, to=2-2]
\end{tikzcd}\]
Consider the natural map
\[
\oplus_{i=1}^nH^0(X,S^i\Omega_X)\to (\oplus_{i=1}^nH^0(C_1,S^i\Omega_{C_1}),\oplus_{i=1}^nH^0(C_2,S^i\Omega_{C_2}))
\]
and let $(b_1,b_2)$ be the image of $b$.
Let $C_{ib_i}$ be the spectral cover of $C_i$ with spectral data $b_i$ for $i=1,2$. Then $C_{ib_i}$ is a closed subscheme of $T_{C_i}^*$ and there is a cartesian diagram
\[\begin{tikzcd}
	{C_{ib_i}} & {Cayley^n(T^*_{C_i}/C_i)} \\
	{C_i} & {Sym^n(T^*_{C_i}/C_i)}.
	\arrow[from=1-1, to=1-2]
	\arrow["{p_i}", from=1-1, to=2-1]
	\arrow["\square"{description}, draw=none, from=1-1, to=2-2]
	\arrow["{p_i'}", from=1-2, to=2-2]
	\arrow["{b_i}", from=2-1, to=2-2]
\end{tikzcd}\]

For any $i\in \Lambda$, let $p_b^{-1}(\mathcal{U}_i)=\mathrm{Spec}\widetilde{A_i[x_{1i},x_{2i}]/I_i}$ and let
\[\tau_i':A_i[x_{1i},x_{2i}]\to A_i[x_{1i},x_{2i}],x_{1i}\mapsto -x_{1i},x_{2i}\mapsto -x_{2i}\]
be the $A_i$-algebra isomorphism.
Note that $\tau_i'(I_i)=I_i$. This induces the isomorphism $\tau_i':A_i[x_{1i},x_{2i}]/I_i\to A_i[x_{1i},x_{2i}]/I_i$.
Let $\tau_i:p_b^{-1}(\mathcal{U}_i)\to p_b^{-1}(\mathcal{U}_i)$ be the induced morphism. 
Note that $\tau_i$ is $GL_2$-equivariant for any $i\in \Lambda$.
By the glueing of $\tau_i$, we obtain the isomorphism $\tau:X_b\to X_b$.

Let $X_1,\dots,X_r$ be the irreducible components of $X_{b}$. 
Then the action of $\tau$ permutes the set $\{X_1,\dots,X_r\}$.
Let 
\[
\mathcal{A}_0=\{i\mid 1\leq i\leq r,\tau(X_i)=X_i,\exists x\in X_i,\dim\mathscr{O}_{X_i,x}=1,\tau(x)=x\},
\]
\[
\mathcal{A}_1=\{(i,j)\mid 1\leq i<j\leq r,\tau(X_i)= X_j\},
\]
\[
\mathcal{A}_2=\{i\mid 1\leq i\leq r,\tau(X_i)=X_i,\nexists x\in X_i,\dim\mathscr{O}_{X_i,x}=1,\tau(x)=x\}.
\]
Let $X_i^{nor}\to X_i$ be the normalization morphism for any $1\leq i\leq r$.
For $i\in \mathcal{A}_0$, let $\pi_i:Y_i=X_i\to X_b$ be the closed immersion.
For $i=(i_1,i_2)\in \mathcal{A}_1$, let $\pi_i:Y_i=X_{i_1}^{nor}\bigsqcup X_{i_2}^{nor}\to X_{b}\to X$ be the natural morphism.
For $i\in \mathcal{A}_2$, let $\pi_i:Y_i=X_i^{nor}\to X_{b}\to X$ be the natural morphism.
Note that $\tau$ can extend to an involution $\hat{\tau}_i$ on $Y_i$.
Since $\pi_i\circ \hat{\tau}_i=\pi_i$ for $i\in \mathcal{A}_1\cup\mathcal{A}_2$, we have $\pi_i=\bar{\pi}_i\circ \hat{\pi}_i$ where $\hat{\pi}_i:Y_i\to \bar{Y}_i=Y_i/<\hat{\tau}_i>$ and $\bar{\pi}_i:\bar{Y}_i\to X$.
Since $X_i^{nor}$ is normal and $<\hat{\tau}_i>$ is a finite group, we see that $\bar{Y}_i$ is irreducible normal.

\begin{lemma}\label{lemma: A_0 case}
For $i\in \mathcal{A}_0$, there is $(\mathscr{E}_i={p_b}_*{\pi_i}_*\mathscr{F}_i,\theta_i,Q_i)$, where  $\mathscr{F}_i$ is a maximal Cohen--Macaulay $\mathscr{O}_{Y_i}$-module of generic rank $1$, $\theta_{i}:\mathscr{T}_X\to {p_b}_*\mathscr{O}_{X_{b}}\to {\pi_i}_*\mathscr{O}_{Y_{i}}\to \mathscr{E}nd_{\mathscr{O}_X}(\mathscr{E}_i)$, and $Q_i:\mathscr{E}_i\otimes \mathscr{E}_i\to \mathscr{O}_X$ is a perfect alternating bilinear form such that $\theta_i$ is anti-self-adjoint with respect to $Q_i$. 
\end{lemma}
\begin{proof}
By \cite[Lemma 2.4]{hausel2012prym}, the underlying reduced curve of each irreducible component of $C_{jb_j}$ is again a spectral cover for $j=1,2$.
By Lemma \ref{lem: X_b to C_1bC_2b}, we can assume $X_i\cong C_{1b_{1i}}\times_{\mathbb{K}}C_{2b_{2i}}$, where the irreducible component $C_{jb_{ji}}$ of $C_{jb_j}$ is a spectral cover of $C_j$ with spectral data $b_{ji}$ for $j=1,2$.
Since $\tau(X_i)=X_i$ and $X_i$ is irreducible, and there is $x\in X_i$ with $\dim\mathscr{O}_{X_i,x}=1$ such that $\tau(x)=x$, we see that $b_{1i}=0$ or $b_{2i}=0$. We can assume $b_{2i}=0$. Then $b_{1i}\in \oplus_{j=1}^mH^0(C_1,S^{2j}\Omega_{C_1})$ where $m=[K(C_{1b_{1i}}):K(C_1)]$.
Let $q_i:C_{1b_{1i}}\to C_1$ be the natural map.
Since $C_1$ is a smooth projective curve and $C_{1b_{1i}}$ is irreducible, there is $(\mathscr{E}_i={q_i}_*\mathscr{F}_i',\theta_i',Q_i')$, where  $\mathscr{F}_i'$ is a maximal Cohen--Macaulay $\mathscr{O}_{C_{1b_{1i}}}$-module of generic rank $1$, $\theta_{i}':\mathscr{T}_{C_1}\to {q_i}_*\mathscr{O}_{C_{1b_{1i}}}\to \mathscr{E}nd_{\mathscr{O}_{C_1}}(\mathscr{E}_i')$, and $Q_i':\mathscr{E}_i'\otimes \mathscr{E}_i'\to \mathscr{O}_{C_1}$ is a perfect alternating bilinear form such that $\theta_i'$ is anti-self-adjoint with respect to $Q_i'$. 
Since $b_{2i}=0$, there is a cartesian diagram 
\[\begin{tikzcd}
	{X_i} & {C_{1b_{1i}}} \\
	{X} & {C_1}.
	\arrow["{q_i'}", from=1-1, to=1-2]
	\arrow["{p_b\circ \pi_i}", from=1-1, to=2-1]
	\arrow["\square"{description}, draw=none, from=1-1, to=2-2]
	\arrow["{q_i}", from=1-2, to=2-2]
	\arrow["{q}", from=2-1, to=2-2]
\end{tikzcd}\]
Let $\mathscr{F}_i={q_i'}_*\mathscr{F}_i',\mathscr{E}_i={p_b}_*{\pi_i}_*\mathscr{F}_i$, and $\theta_i,Q_i$ be induced by $\theta_i', Q_i'$.
Since $q_i'$ and $q$ are flat,  we see that $\mathscr{F}_i$ is a maximal Cohen--Macaulay $\mathscr{O}_{X_i}$-module of generic rank $1$ and $Q_i:\mathscr{E}_i\otimes \mathscr{E}_i\to \mathscr{O}_X$ is a perfect alternating bilinear form such that $\theta_i$ is anti-self-adjoint with respect to $Q_i$. 

\end{proof}

\begin{lemma}\label{lemma: A_1 case}
For $i\in \mathcal{A}_1$, there is $(\mathscr{E}_i={\pi_i}_*\mathscr{F}_i,\theta_i,Q_i)$, where  $\mathscr{F}_i$ is a maximal Cohen--Macaulay $\mathscr{O}_{Y_i}$-module of generic rank $1$, $\theta_{i}:\mathscr{T}_X\to {p_b}_*\mathscr{O}_{X_{b}}\to {\pi_i}_*\mathscr{O}_{Y_{i}}\to \mathscr{E}nd_{\mathscr{O}_X}(\mathscr{E}_i)$, and $Q_i:\mathscr{E}_i\otimes \mathscr{E}_i\to \mathscr{O}_X$ is a perfect alternating bilinear form such that $\theta_i$ is anti-self-adjoint with respect to $Q_i$. 
\end{lemma}
\begin{proof}
Let $i=(i_1,i_2)$ and $Z_j=X_{i_j}^{nor}$ for $j=1,2$.
We may assume that $Z_1=Z_2$ and the action of $\hat{\tau}_i$ on $Y_i=Z_1\sqcup Z_2$ is the involution that swaps the two components. Then $\bar{Y}_i\cong Z_1$.
By Lemma \ref{lem: Y_i=C times C' sm}, $Y_i$ and $\bar{Y}_i$ are smooth projective over $\mathbb{K}$. 
Thus, $\bar{\pi}_i$ is a flat projective local complete intersection and $\hat{\pi}_i$ is \'etale projective. 
By Lemma \ref{lemma: dualizing sheaf},  $\omega_{\bar{\pi}_i}$ is invertible and $\omega_{\pi_i}\cong {\hat{\pi}_i}^*\omega_{\bar{\pi}_i}$. 
Thus, we may assume that $\omega_{\pi_i}={q_1}_*L_1\oplus {q_2}_*L_2$ where $L=L_j$ is an invertible sheaf on $Z_j$ and $q_j:Z_j\to Y_i$ is the natural map for $j=1,2$.
Let $\mathscr{F}_i={q_1}_*L_1\oplus {q_2}_*\mathscr{O}_{Z_2}$ be the invertible sheaf on $Y_i$.
Then $\hat{\tau}_i^*\mathscr{F}_i={q_1}_*\mathscr{O}_{Z_1} \oplus {q_2}_*L_2 $ and $\hat{\tau}_i^*\mathscr{F}_i\otimes\mathscr{F}_i\cong  \omega_{\pi_i}$.
Let $\phi_1:\hat{\tau}_i^*\mathscr{F}_i|_{Z_1}\to \mathscr{H}om_{\mathscr{O}_{Y_i}}(\mathscr{F}_i,\omega_{\pi_i})|_{Z_1}$ be the isomorphism induced by the natural isomorphism $\varphi_1:\mathscr{O}_{Z_1}\to \mathscr{H}om_{\mathscr{O}_{Z_1}}(L_1,L_1)$, and $\phi_2:\hat{\tau}_i^*\mathscr{F}_i|_{Z_2}\to \mathscr{H}om_{\mathscr{O}_{Y_i}}(\mathscr{F}_i,\omega_{\pi_i})|_{Z_2}$ be the isomorphism induced by the isomorphism $\varphi_2:L_2\to \mathscr{H}om_{\mathscr{O}_{Z_2}}(\mathscr{O}_{Z_2},L_2)$ such that following diagrams
\[\begin{tikzcd}
	\mathscr{H}om_{\mathscr{O}_{Z_2}}(\mathscr{H}om_{\mathscr{O}_{Z_2}}(\mathscr{O}_{Z_2},L_2),L_2) & \mathscr{H}om_{\mathscr{O}_{Z_2}}(L_2,L_2) \\
	\mathscr{O}_{Z_2}, &
	\arrow["{\varphi_2^*}", from=1-1, to=1-2]
	\arrow["{\rho_2}", from=1-1, to=2-1]
	\arrow["{-\varphi_1}"', from=2-1, to=1-2]
\end{tikzcd}\]
\[\begin{tikzcd}
	\mathscr{H}om_{\mathscr{O}_{Z_1}}(\mathscr{H}om_{\mathscr{O}_{Z_1}}(L_1,L_1),L_1) & \mathscr{H}om_{\mathscr{O}_{Z_1}}(\mathscr{O}_{Z_1},L_1) \\
	L_1.&
	\arrow["{\varphi_1^*}", from=1-1, to=1-2]
	\arrow["{\rho_1}", from=1-1, to=2-1]
	\arrow["{-\varphi_2}"', from=2-1, to=1-2]
\end{tikzcd}\]
are commutative, where $\varphi_j^*=\mathscr{H}om_{\mathscr{O}_{Z_j}}(-,L_j)(\varphi_j)$ and $\rho_j$ is the natural isomorphism for $j=1,2$.
We obtain the isomorphism $\phi:\hat{\tau}_i^*\mathscr{F}_i\to \mathscr{H}om_{\mathscr{O}_{Y_i}}(\mathscr{F}_i,\omega_{\pi_i})$ such that $\phi|_{Z_1}=\phi_1$ and $\phi|_{Z_2}=\phi_2$.
Hence we have the commutative diagram 
\[\begin{tikzcd}
	{\mathscr{H}om_{\mathscr{O}_{Y_i}}(\hat{\tau}_i^*\mathscr{H}om_{\mathscr{O}_{Y_i}}(\mathscr{F}_i,\omega_{\pi_i}),\omega_{\pi_i})} & {\hat{\tau}_i^*\mathscr{F}_i} \\
	{\mathscr{H}om_{\mathscr{O}_{Y_i}}(\hat{\tau}_i^*\hat{\tau}_i^*\mathscr{F}_i,\omega_{\pi_i})} & { \mathscr{H}om_{\mathscr{O}_{Y_i}}(\mathscr{F}_i,\omega_{\pi_i})}.
	\arrow["{\sim}", from=1-1, to=1-2]
	\arrow["{(\hat{\tau}_i^*\phi)^*:=\mathscr{H}om_{\mathscr{O}_{Y_i}}(-,\omega_{\pi_i})(\hat{\tau}_i^*\phi)}", from=1-1, to=2-1]
	\arrow["{-\phi}", from=1-2, to=2-2]
	\arrow["{\sim}", from=2-1, to=2-2]
\end{tikzcd}\]

Let $\mathscr{E}_i={\pi_i}_*\mathscr{F}_i$ and $\theta_{i}:\mathscr{T}_X\to {p_b}_*\mathscr{O}_{X_{b}}\to {\pi_i}_*\mathscr{O}_{Y_{i}}\to \mathscr{E}nd_{\mathscr{O}_X}(\mathscr{E}_i)$.
Since $Y_i$ is regular, the $\mathscr{O}_{Y_i}$-module $\mathscr{F}_i$ is maximal Cohen--Macaulay of generic rank $1$ and $\mathscr{E}_i$ is locally free. 
By Lemma \ref{lemma: dualizing sheaf}, there is the natural isomorphism
\begin{align*}
    \Phi_i: \ &\mathscr{E}_i={\pi_i}_*\mathscr{F}_i\xrightarrow{\sim}{\pi_i}_*(\hat{\tau}_{i*}\hat{\tau}_i^*\mathscr{F}_i)\xrightarrow{\sim}{\pi_i}_*\hat{\tau}_i^*\mathscr{F}_i\\
    &\xrightarrow{{\pi_i}_*\phi}{\pi_i}_*\mathscr{H}om_{\mathscr{O}_{Y_i}}(\mathscr{F}_i,\omega_{\pi_i})\xrightarrow{\sim}\mathscr{H}om_{\mathscr{O}_{X}}({\pi_i}_*\mathscr{F}_i,\mathscr{O}_{X})=\mathscr{E}_i^{\vee}.\\
\end{align*}
Hence the form $Q_i:\mathscr{E}_i\otimes_{\mathscr{O}_X}\mathscr{E}_i\to \mathscr{O}_X$ induced by $\Phi_i$ is perfect bilinear.
Next, we claim that $Q_i$ is alternating.
Consider the isomorphism
\begin{align*}
 \Phi_i^*&:\mathscr{H}om_{\mathscr{O}_X}(\mathscr{E}_i^{\vee},\mathscr{O}_X)\xrightarrow{\sim}\mathscr{H}om_{\mathscr{O}_X}({\pi_i}_*\mathscr{H}om_{\mathscr{O}_{Y_i}}(\mathscr{F}_i,\omega_{\pi_i}),\mathscr{O}_X)\\
&\xrightarrow{({\pi_i}_*\phi)^*}\mathscr{H}om_{\mathscr{O}_X}({\pi_i}_*\hat{\tau}_i^*\mathscr{F}_i,\mathscr{O}_X)\xrightarrow{\sim}\mathscr{H}om_{\mathscr{O}_X}({\pi_i}_*\hat{\tau}_{i*}\hat{\tau}_i^*\mathscr{F}_i,\mathscr{O}_X)\xrightarrow{\sim}\mathscr{E}_i^{\vee}.
\end{align*}
By Lemma \ref{lemma: dualizing sheaf}, one may verify that the following diagram commutes:
\[\begin{tikzcd}
	{\mathscr{H}om_{\mathscr{O}_X}({\pi_i}_*\mathscr{H}om_{\mathscr{O}_{Y_i}}(\mathscr{F}_i,\omega_{\pi_i}),\mathscr{O}_X)} & {\mathscr{H}om_{\mathscr{O}_X}({\pi_i}_*\hat{\tau}_i^*\mathscr{F}_i,\mathscr{O}_X)} \\
	{\mathscr{H}om_{\mathscr{O}_X}({\pi_i}_*\hat{\tau}_{i*}\hat{\tau}_i^*\mathscr{H}om_{\mathscr{O}_{Y_i}}(\mathscr{F}_i,\omega_{\pi_i}),\mathscr{O}_X)} & { \mathscr{H}om_{\mathscr{O}_{X}}({\pi_i}_*\hat{\tau}_{i*}\hat{\tau}_i^*\hat{\tau}_i^*\mathscr{F}_i,\mathscr{O}_{X})}\\
    {\pi_i}_*{\mathscr{H}om_{\mathscr{O}_{Y_i}}(\hat{\tau}_i^*\mathscr{H}om_{\mathscr{O}_{Y_i}}(\mathscr{F}_i,\omega_{\pi_i}),\omega_{\pi_i})} & { {\pi_i}_*\mathscr{H}om_{\mathscr{O}_{Y_i}}(\hat{\tau}_i^*\hat{\tau}_i^*\mathscr{F}_i,\omega_{\pi_i})}\\
    {\pi_i}_*\hat{\tau}_i^*\mathscr{F}_i
    &{\pi_i}_*\mathscr{H}om_{\mathscr{O}_{Y_i}}(\mathscr{F}_i,\omega_{\pi_i}).
	\arrow["{({\pi_i}_*\phi)^*}", from=1-1, to=1-2]
	\arrow["{\sim}", from=1-1, to=2-1]
	\arrow["{\sim}", from=1-2, to=2-2]
	\arrow["{\sim}", from=2-1, to=2-2]
    \arrow["{\sim}", from=2-1, to=3-1]
    \arrow["{\sim}", from=2-2, to=3-2]
    \arrow["{{\pi_i}_*(\hat{\tau}_i^*\phi)^*}", from=3-1, to=3-2]
    \arrow["{\sim}", from=3-1, to=4-1]
    \arrow["{\sim}", from=3-2, to=4-2]
    \arrow["{-{\pi_i}_*\phi}", from=4-1, to=4-2]
\end{tikzcd}\]
Hence we obtain the the commutative diagram 
\[\begin{tikzcd}
	\mathscr{E}_i & \mathscr{H}om_{\mathscr{O}_X}(\mathscr{E}_i^{\vee},\mathscr{O}_X) \\
	\mathscr{E}_i^{\vee}. & 
	\arrow["{\sim}", from=1-1, to=1-2]
	\arrow["{-\Phi_i}", from=1-1, to=2-1]
	\arrow["{\Phi_i^*}", from=1-2, to=2-1]
\end{tikzcd}\]
So $Q_i$ is alternating.

Finally, we claim that $\theta_i$ is anti-self-adjoint with respect to $Q_i$. 
Note that $\Phi_i$ induces the isomorphism 
\[\Phi_i':{\pi_i}_*\hat{\tau}_i^*\mathscr{F}_i
    \xrightarrow{{\pi_i}_*\phi}{\pi_i}_*\mathscr{H}om_{\mathscr{O}_{Y_i}}(\mathscr{F}_i,\omega_{\pi_i})\xrightarrow{\sim}\mathscr{H}om_{\mathscr{O}_{X}}({\pi_i}_*\mathscr{F}_i,\mathscr{O}_{X})=\mathscr{E}_i^{\vee}
\]
of ${\pi_i}_*\mathscr{O}_{Y_i}$-modules, given ${\pi_i}_*\hat{\tau}_i^*\mathscr{F}_i$ a structure of ${\pi_i}_*\mathscr{O}_{Y_i}$-module by ${\pi_i}_*\hat{\tau}_i^*\mathscr{O}_{Y_i}\cong {\pi_i}_*\mathscr{O}_{Y_i}$ and $\mathscr{E}_i^{\vee}$ a natural structure of ${\pi_i}_*\mathscr{O}_{Y_i}$-module.
Since we have the isomorphism ${\pi_i}_*\hat{\tau}_i^\#:{\pi_i}_*\mathscr{O}_{Y_i}\to {\pi_i}_*\hat{\tau}_{i*}\mathscr{O}_{Y_i}\cong {\pi_i}_*\mathscr{O}_{Y_i}$, this gives $E_i':={\pi_i}_*\mathscr{F}_i$ a structure of ${\pi_i}_*\mathscr{O}_{Y_i}$-module.
One may verify that $E_i'\to {\pi_i}_*\hat{\tau}_i^*\mathscr{F}_i$ is a isomorphism of ${\pi_i}_*\mathscr{O}_{Y_i}$-modules.
Then $\Phi_i:E_i'\to \mathscr{E}_i^{\vee}$ is a isomorphism of ${\pi_i}_*\mathscr{O}_{Y_i}$-modules.
Hence $\theta_i$ is anti-self-adjoint with respect to $Q_i$. 
This completes the proof.
\end{proof}

For $i\in \mathcal{A}_2$, $X_i$ has no fixed prime cycle of $\tau$. Then there is an open subset $U$ of $X$ such that $\dim(X\setminus U)=0$ and $W_i:=\pi_i^{-1}(U)$ has no fixed prime cycle of $\hat{\tau}_i$ and $\phi_i':={\hat{\pi}_i}|_{W_i}$ is \'etale for any $i\in \mathcal{A}_2$. Let $W'_i=\bar{\pi}_i^{-1}(U)=W_i/<\hat{\tau}_i>$ and $\phi_i=\pi_i|_{W}$ for $i\in \mathcal{A}_2$.

\begin{lemma}\label{lemma: varphi}
For $i\in \mathcal{A}_2$, there are an invertible sheaf $\mathscr{L}_i$ on $W_i$ and an isomorphism $\varphi_i:\hat{\tau}_i^*\mathscr{L}_i\to \mathscr{H}om_{\mathscr{O}_{W_i}}(\mathscr{L}_i,{\phi_i'}^*L_i')$ where $L_i'$ is an invertible sheaf on $W_i'$ and $\omega_{\phi_i}\cong {\phi_i'}^*L_i'$, such that there is the commutative diagram 
\[\begin{tikzcd}
	{\mathscr{H}om_{\mathscr{O}_{W_i}}(\hat{\tau}_i^*\mathscr{H}om_{\mathscr{O}_{W_i}}(\mathscr{L}_i,{\phi_i'}^*L_i'),{\phi_i'}^*L_i')} & {\hat{\tau}_i^*\mathscr{L}_i} \\
	{\mathscr{H}om_{\mathscr{O}_{W_i}}(\hat{\tau}_i^*\hat{\tau}_i^*\mathscr{L}_i,{\phi_i'}^*L_i')} & { \mathscr{H}om_{\mathscr{O}_{W_i}}(\mathscr{L}_i,{\phi_i'}^*L_i')}.
	\arrow["{\sim}", from=1-1, to=1-2]
	\arrow["{(\hat{\tau}_i^*\varphi_i)^*:=\mathscr{H}om_{\mathscr{O}_{W_i}}(-,{\phi_i'}^*L_i')(\hat{\tau}_i^*\varphi_i)}", from=1-1, to=2-1]
	\arrow["{-\varphi_i}", from=1-2, to=2-2]
	\arrow["{\sim}", from=2-1, to=2-2]
\end{tikzcd}\]
\end{lemma}
\begin{proof}
By Lemma \ref{lemma: pi flat}, $\pi_i$ is finite flat. Then $\pi_i$ and $\bar{\pi}_i$ are finite surjective.
By \cite[III, Ex. 5.7]{hartshorne1977AG}, we can show that
$Y_i$ and $\bar{Y}_i$ are projective over $\mathbb{K}$, then $\pi_i$, $\hat{\pi}_i$ and $\bar{\pi}_i$ are projective.
By Lemma \ref{lem: Y_i=C times C' sm}, $Y_i$ is smooth projective over $\mathbb{K}$. 
Then $\phi_i$ is a flat projective local complete intersection.
Note that $\phi_i'$ is \'etale projective. Then $W_i'$ is regular, then $\bar{\pi}_i|_{W_i'}$ is a flat projective local complete intersection. 
By Lemma \ref{lemma: dualizing sheaf},  $\omega_{\phi_i'}\cong \mathscr{O}_{W_i}$, $L_i':=\omega_{\bar{\pi}_i|_{W_i'}}$ is invertible and $\omega_{\phi_i}\cong {\phi_i'}^*L_i'$. 

By \cite[Exercise 6.4.7]{liu2006algebraic}, there is a homomorphism of $\mathscr{O}_{W_i'}$-modules $\rho:{\phi_i'}_*\mathscr{O}_{W_i}\to {\phi_i'}_*\omega_{\phi_i'}\to \mathscr{O}_{W_i'}$ such that 
if $V=\mathrm{Spec}A'$ is an open subset of $W'_i$, then $\rho(V):A={\phi_i'}_*\mathscr{O}_{W_i'}(V)\to  \mathscr{O}_{W_i'}=A'$ sends $a\in A$ to $\mathrm{Tr}_{A/A'}(a)$.
Since $\phi_i'$ is finite flat of degree $2$, ${\phi_i'}_*\mathscr{O}_{W_i}$ is locally free of rank $2$.
One may verify that $\rho$ induces ${\phi'_i}_*\mathscr{O}_{W_i}\cong \mathscr{O}_{W_i'}\oplus \mathscr{F}_i$ and $\mathscr{F}_i$ is locally free of rank $1$.
By Lemma \ref{lemma: dualizing sheaf}, we have ${\phi'_i}_*\omega_{{\phi'_i}}\cong ({\phi'_i}_*\mathscr{O}_{W_i})^{\vee}$, then $ {\phi'_i}_*\mathscr{O}_{W_i}\cong ({\phi'_i}_*\mathscr{O}_{W_i})^{\vee}$ since  $\omega_{{\phi'_i}}\cong \mathscr{O}_{W_i}$, then $\mathscr{F}_i\cong \det({\phi'_i}_*\mathscr{O}_{W_i})\cong \det({\phi'_i}_*\mathscr{O}_{W_i})^{\vee}$, then $(\mathscr{F}_i)^2\cong \mathscr{O}_{W_i'}$, then ${\phi_i'}^*\mathscr{F}_i\cong \mathscr{O}_{W_i}$ since $\phi'_i$ is \'etale.

We can take a cycle $D_i',D_i''\in Z^1(W_i')$ and a cycle $D_i\in Z^1(W_i)$ such that $L_i'=\mathscr{O}_{W_i'}(D_i')$, $\mathscr{F}_i=\mathscr{O}_{W_i'}(D_i'')$ and ${\phi_i'}_*D_i=D_i'+D_i''$. Since $W_i$ has no fixed prime cycle of $\hat{\tau}_i$, $\hat{\tau}_i^*(D_i)+D_i={\phi_i'}^*(D'_i+D_i'')$.
Let $L_i=\mathscr{O}_{W_i}(D_i)$. 
Then $\hat{\tau}_i^*L_i\otimes L_i\cong{\phi_i'}^*(L_i'\otimes \mathscr{F}_i)\cong{\phi_i'}^*L_i'$. 
Let $\rho_{1\hat{\tau}_i}:\hat{\tau}_i^*(\hat{\tau}_i^*L_i\otimes L_i)\xrightarrow{\sim} \hat{\tau}_i^*L_i\otimes L_i$ and $\rho_{2\hat{\tau}_i}:\hat{\tau}_i^*({\phi_i'}^*L_i')\xrightarrow{\sim} {\phi_i'}^*L_i'$ be the natural morphisms. 
One may verify that $\rho_{1\hat{\tau}_i}\circ\hat{\tau}_i^*\rho_{1\hat{\tau}_i}=Id_{\hat{\tau}_i^*L_i\otimes L_i}$ and $\rho_{2\hat{\tau}_i}\circ\hat{\tau}_i^*\rho_{2\hat{\tau}_i}=Id_{{\phi_i'}^*L_i'}$. Let $H_i=<\hat{\tau}_i>$.
Then $\rho_{1\hat{\tau}_i}$ induces a $H_i$-linearization of $\hat{\tau}_i^*L_i\otimes L_i$ and $\rho_{2\hat{\tau}_i}$ induces a $H_i$-linearization of ${\phi_i'}^*L'_i$.
This induces a $H_i$-linearization $\rho_{\hat{\tau}_i}$ of $\mathscr{H}om_{\mathscr{O}_{W_i}}(\hat{\tau}_i^*L_i\otimes L_i,{\phi_i'}^*L_i')$.
Then it induces a $\mathbb{K}H_i$-module  $Hom_{\mathscr{O}_{W_i}}(\hat{\tau}_i^*L_i\otimes L_i,{\phi_i'}^*L_i')$ such that 
\[
\hat{\tau}_i.f=\rho_{2\hat{\tau}_i}\circ (\hat{\tau}_i^*f)\circ \rho_{1\hat{\tau}_i}^{-1},\ \  f\in Hom_{\mathscr{O}_{W_i}}(\hat{\tau}_i^*L_i\otimes L_i,{\phi'_i}^*L_i').
\]
We can take an isomorphim 
\[\rho_i':\mathscr{H}om_{\mathscr{O}_{W_i}}(\hat{\tau}_i^*L_i\otimes L_i,{\phi_i'}^*L_i')\to {\phi'_i}^*\mathscr{H}om_{\mathscr{O}_{W_i'}}(\mathscr{F}_i\otimes L_i',L_i').
\]
Let $\rho'_{\hat{\tau}_i}:\hat{\tau}_i^*({\phi'_i}^*\mathscr{H}om_{\mathscr{O}_{W_i'}}(\mathscr{F}_i\otimes L_i',L_i'))\xrightarrow{\sim} {\phi'_i}^*\mathscr{H}om_{\mathscr{O}_{W_i'}}(\mathscr{F}_i\otimes L_i',L_i')$ be the natural morphisms.
One may verify that $\rho'_{\hat{\tau}_i}\circ\hat{\tau}_i^*\rho'_{\hat{\tau}_i}=Id_{{\phi'_i}^*\mathscr{H}om_{\mathscr{O}_{W_i'}}(\mathscr{F}_i\otimes L_i',L_i')}$ and we have the commutative diagram 
\[\begin{tikzcd}
	{\hat{\tau}^*\mathscr{H}om_{\mathscr{O}_{W_i}}(\hat{\tau}_i^*L_i\otimes L_i,{\phi_i'}^*L_i')} & {\hat{\tau}^*({\phi'_i}^*\mathscr{H}om_{\mathscr{O}_{W_i'}}(\mathscr{F}_i\otimes L_i',L_i'))} \\
	{\mathscr{H}om_{\mathscr{O}_{W_i}}(\hat{\tau}_i^*L_i\otimes L_i,{\phi_i'}^*L_i')} & {{\phi'_i}^*\mathscr{H}om_{\mathscr{O}_{W_i'}}(\mathscr{F}_i\otimes L_i',L_i')}.
	\arrow["{\hat{\tau}_i^*\rho_i'}", from=1-1, to=1-2]
	\arrow["{\rho_{\hat{\tau}_i}}", from=1-1, to=2-1]
	\arrow["{\rho'_{\hat{\tau}_i}}", from=1-2, to=2-2]
	\arrow["{\rho_i'}", from=2-1, to=2-2]
\end{tikzcd}\]
Then $\rho'_{\hat{\tau}_i}$ induces a $H_i$-linearization of ${\phi'_i}^*\mathscr{H}om_{\mathscr{O}_{W_i'}}(\mathscr{F}_i\otimes L_i',L_i')$.
This induces $\mathbb{K}H_i$-modules $H^0(W_i,{\phi'_i}^*\mathscr{H}om_{\mathscr{O}_{W_i'}}(\mathscr{F}_i\otimes L_i',L_i'))$ and an isomorphism 
\[
\rho_i'(W_i):Hom_{\mathscr{O}_{W_i}}(\hat{\tau}_i^*L_i\otimes L_i,{\phi_i'}^*L_i')\to H^0(W_i,{\phi'_i}^*\mathscr{H}om_{\mathscr{O}_{W_i'}}(\mathscr{F}_i\otimes L_i',L_i'))
\]
of $\mathbb{K}H_i$-modules.
Let 
\[
H_{i}^{-}=\{f\in Hom_{\mathscr{O}_{W_i}}(\hat{\tau}_i^*L_i\otimes L_i,{\phi_i'}^*L_i')\mid \hat{\tau}_i.f=-f\}.
\]
Note that the $\mathbb{K}H_i$-module $H^0(W_i,{\phi'_i}^*\mathscr{H}om_{\mathscr{O}_{W_i'}}(\mathscr{F}_i\otimes L_i',L_i'))$ induces a $\mathbb{K}H_i$-module $H^0(W_i',{\phi_i'}_*{\phi'_i}^*\mathscr{H}om_{\mathscr{O}_{W_i'}}(\mathscr{F}_i\otimes L_i',L_i'))$ and an isomorphism $$H^0(W_i,{\phi'_i}^*\mathscr{H}om_{\mathscr{O}_{W_i'}}(\mathscr{F}_i\otimes L_i',L_i'))\to H^0(W_i',{\phi_i'}_*{\phi'_i}^*\mathscr{H}om_{\mathscr{O}_{W_i'}}(\mathscr{F}_i\otimes L_i',L_i')).$$
Let $H_i^{-'}$ be the image of $\rho_i'(W_i)(H_i^-)$ under this isomorphism.
One may verify that the natural map $\mathscr{H}om_{\mathscr{O}_{W_i'}}(\mathscr{F}_i\otimes L_i',L_i')\to {\phi_i'}_*{\phi_i'}^*\mathscr{H}om_{\mathscr{O}_{W_i'}}(\mathscr{F}_i\otimes L_i',L_i')$ induces an isomorphism $$H^0(W_i',\mathscr{H}om_{\mathscr{O}_{W_i'}}(\mathscr{F}_i\otimes L_i',L_i'))\cong H^0(W_i',{\phi_i'}_*{\phi_i'}^*\mathscr{H}om_{\mathscr{O}_{W_i'}}(\mathscr{F}_i\otimes L_i',L_i'))^H.$$
Since ${\phi_i'}_*\mathscr{O}_{W_i}\cong \mathscr{O}_{W_i'}\oplus \mathscr{F}_i$, we have ${\phi_i'}_*{\phi_i'}^*\mathscr{H}om_{\mathscr{O}_{W_i'}}(\mathscr{F}_i\otimes L_i',L_i')\cong {\phi_i'}_*\mathscr{O}_{W_i}\otimes \mathscr{F}_i^{\vee}\cong \mathscr{O}_{W_i'}\oplus \mathscr{F}_i^{\vee}$, thus, as $\mathbb{K}$-modules, we have
\begin{align*}
    H_i^-\cong H_i^{-'}&\cong \frac{H^0(W_i',{\phi_i'}_*{\phi_i'}^*\mathscr{H}om_{\mathscr{O}_{W_i'}}(\mathscr{F}_i\otimes L_i',L_i'))}{H^0(W_i',{\phi_i'}_*{\phi_i'}^*\mathscr{H}om_{\mathscr{O}_{W_i'}}(\mathscr{F}_i\otimes L_i',L_i'))^{H_i}}\\
   & \cong \frac{H^0(W_i',\mathscr{O}_{W_i'}\oplus \mathscr{F}_i^{\vee})}{H^0(W_i',\mathscr{F}_i^{\vee})}\\
   & \cong H^0(W_i',\mathscr{O}_{W_i'}).
\end{align*}
Since $Y_i,\bar{Y}_i$ are normal projective irreducible surfaces over $\mathbb{K}$ and $\dim(Y_i\setminus W_i)=\dim(\bar{Y}_i\setminus W_i')=0$ and $\hat{\tau}_i^*L_i\otimes L_i\cong {\phi_i'}^*L_i'$, we have 
\[
Hom_{\mathscr{O}_{W_i}}(\hat{\tau}_i^*L_i\otimes L_i,{\phi_i'}^*L_i')\cong H^0(W_i,\mathscr{O}_{W_i})\cong H^0(Y_i,\mathscr{O}_{Y_i})\cong \mathbb{K},
\]\[
H^0(W_i',\mathscr{O}_{W_i'})\cong H^0(\bar{Y}_i,\mathscr{O}_{\bar{Y}_i})\cong \mathbb{K}.
\]
Then $Hom_{\mathscr{O}_{W_i}}(\hat{\tau}_i^*L_i\otimes L_i,{\phi_i'}^*L_i')=H_i^-$.
So we can take an isomorphism $f_i\in H_i^-$. Note that $f_i$ induces an isomorphism $\varphi_i:\hat{\tau}_i^*L_i\to \mathscr{H}om_{\mathscr{O}_{W_i}}(L_i,{\phi_i'}^*L_i')$. Since $\hat{\tau}_i.f_i=-f_i$, we have the commutative diagram 
\[\begin{tikzcd}
	{\mathscr{H}om_{\mathscr{O}_{W_i}}(\hat{\tau}_i^*\mathscr{H}om_{\mathscr{O}_{W_i}}(L_i,{\phi_i'}^*L_i'),{\phi_i'}^*L_i')} & {\hat{\tau}_i^*L_i} \\
	{\mathscr{H}om_{\mathscr{O}_{W_i}}(\hat{\tau}_i^*\hat{\tau}_i^*L_i,{\phi_i'}^*L_i')} & { \mathscr{H}om_{\mathscr{O}_{W_i}}(L_i,{\phi_i'}^*L_i')}.
	\arrow["{\sim}", from=1-1, to=1-2]
	\arrow["{(\hat{\tau}_i^*\varphi_i)^*}", from=1-1, to=2-1]
	\arrow["{-\varphi_i}", from=1-2, to=2-2]
	\arrow["{\sim}", from=2-1, to=2-2]
\end{tikzcd}\]
\end{proof}

\begin{lemma}\label{lemma: L_i to tau_*tau^*L_i}
The canonical morphism $\mathscr{F}\to \hat{\tau}_{i*}\hat{\tau}_i^*\mathscr{F}$ is an isomorphism for $i\in \mathcal{A}_2$ and any coherent sheaf $\mathscr{F}$ on $W_i$.
\end{lemma}
\begin{proof}
For any affine open subset $\mathrm{Spec}R$ of $W_i$ with action of $\hat{\tau}_i$, let $\mathscr{F}|_{\mathrm{Spec}R}=\tilde{M}$, and $f: R\to R':=R$ be the isomorphism induced by $\hat{\tau}_i$.
Then $\mathscr{F}|_{\mathrm{Spec}R}\to (\hat{\tau}_{i*}\hat{\tau}_i^*\mathscr{F})|_{\mathrm{Spec}R}$ induces a morphism $g:M\to M\otimes_R R'$ of $R$-modules, where $M\otimes_R R'$ has the $R$-module structure induced by $f$. One may verify that $g$ is an isomorphism. Hence $\mathscr{F}\to \hat{\tau}_{i*}\hat{\tau}_i^*\mathscr{F}$ is an isomorphism.
\end{proof}

For $i\in \mathcal{A}_2$, since $\phi_i\circ \hat{\tau}_i=\phi_i$, we have ${\phi_i}_* \circ  \hat{\tau}_{i*}\cong {\phi_i}_*$. 
By Lemma \ref{lemma: dualizing sheaf}, Lemma \ref{lemma: varphi} and Lemma \ref{lemma: L_i to tau_*tau^*L_i}, we have the natural isomorphism
\begin{align*}
    \Phi_i: \ &{\phi_i}_*\mathscr{L}_i\xrightarrow{\sim}{\phi_i}_*(\hat{\tau}_{i*}\hat{\tau}_i^*\mathscr{L}_i)\xrightarrow{\sim}{\phi_i}_*\hat{\tau}_i^*\mathscr{L}_i\\
    &\xrightarrow{{\phi_i}_*\varphi_i}{\phi_i}_*\mathscr{H}om_{\mathscr{O}_{W_i}}(\mathscr{L}_i,\phi_i'^*L_i')\xrightarrow{\sim}\mathscr{H}om_{\mathscr{O}_{U}}({\phi_i}_*\mathscr{L}_i,\mathscr{O}_{U}).
\end{align*}

\begin{lemma}\label{lemma: A_2' case}
For $i\in \mathcal{A}_2$, the form $Q_i':{\phi_i}_*\mathscr{L}_i\otimes_{\mathscr{O}_U}{\phi_i}_*\mathscr{L}_i\to \mathscr{O}_U$ induced by $\Phi_i$ is a perfect alternating bilinear form such that $\theta_i'$ is anti-self-adjoint with respect to $Q_i'$, where $\theta_{i}':\mathscr{T}_U\to ({p_b}_*\mathscr{O}_{X_{b}})|_U\to ({\pi_i}_*\mathscr{O}_{X_i})|_U\to \mathscr{E}nd_{\mathscr{O}_U}({\phi_i}_*\mathscr{L}_i)$. 
\end{lemma}
\begin{proof}
Since $\Phi_i:{\phi_i}_*\mathscr{L}_i\to \mathscr{H}om_{\mathscr{O}_{U}}({\phi_i}_*\mathscr{L}_i,\mathscr{O}_{U})$ is an isomorphism, $Q_i'$ is a perfect bilinear form.
Next, we claim that $Q_i'$ is alternating. Let $\mathscr{E}_i={\phi_i}_*\mathscr{L}_i$. 
Consider the isomorphism
\begin{align*}
 \Phi_i^*&:\mathscr{H}om_{\mathscr{O}_U}(\mathscr{E}_i^{\vee},\mathscr{O}_U)\xrightarrow{\sim}\mathscr{H}om_{\mathscr{O}_U}({\phi_i}_*\mathscr{H}om_{\mathscr{O}_{W_i}}(\mathscr{L}_i,\phi_i'^*L_i'),\mathscr{O}_U)\\
&\xrightarrow{({\phi_i}_*\varphi_i)^*}\mathscr{H}om_{\mathscr{O}_U}({\phi_i}_*\hat{\tau}_i^*\mathscr{L}_i,\mathscr{O}_U)\xrightarrow{\sim}\mathscr{H}om_{\mathscr{O}_U}({\phi_i}_*\hat{\tau}_{i*}\hat{\tau}_i^*\mathscr{L}_i,\mathscr{O}_U)\xrightarrow{\sim}\mathscr{E}_i^{\vee}.
\end{align*}
By Lemma \ref{lemma: dualizing sheaf}, Lemma \ref{lemma: varphi} and Lemma \ref{lemma: L_i to tau_*tau^*L_i}, one may verify that the following diagram commutes:
\[\begin{tikzcd}
	{\mathscr{H}om_{\mathscr{O}_U}({\phi_i}_*\mathscr{H}om_{\mathscr{O}_{W_i}}(\mathscr{L}_i,\phi_i'^*L_i'),\mathscr{O}_U)} & {\mathscr{H}om_{\mathscr{O}_U}({\phi_i}_*\hat{\tau}_i^*\mathscr{L}_i,\mathscr{O}_U)} \\
	{\mathscr{H}om_{\mathscr{O}_U}({\phi_i}_*\hat{\tau}_{i*}\hat{\tau}_i^*\mathscr{H}om_{\mathscr{O}_{W_i}}(\mathscr{L}_i,\phi_i'^*L_i'),\mathscr{O}_U)} & { \mathscr{H}om_{\mathscr{O}_{U}}({\phi_i}_*\hat{\tau}_{i*}\hat{\tau}_i^*\hat{\tau}_i^*\mathscr{F}_i,\mathscr{O}_{U})}\\
    {\phi_i}_*{\mathscr{H}om_{\mathscr{O}_{W_i}}(\hat{\tau}_i^*\mathscr{H}om_{\mathscr{O}_{W_i}}(\mathscr{L}_i,\phi_i'^*L_i'),\phi_i'^*L_i')} & { {\phi_i}_*\mathscr{H}om_{\mathscr{O}_{W_i}}(\hat{\tau}_i^*\hat{\tau}_i^*\mathscr{L}_i,\omega_{\pi_i})}\\
    {\phi_i}_*\hat{\tau}_i^*\mathscr{L}_i
    &{\phi_i}_*\mathscr{H}om_{\mathscr{O}_{W_i}}(\mathscr{L}_i,\phi_i'^*L_i').
	\arrow["{({\phi_i}_*\varphi_i)^*}", from=1-1, to=1-2]
	\arrow["{\sim}", from=1-1, to=2-1]
	\arrow["{\sim}", from=1-2, to=2-2]
	\arrow["{\sim}", from=2-1, to=2-2]
    \arrow["{\sim}", from=2-1, to=3-1]
    \arrow["{\sim}", from=2-2, to=3-2]
    \arrow["{{\phi_i}_*(\hat{\tau}_i^*\varphi_i)^*}", from=3-1, to=3-2]
    \arrow["{\sim}", from=3-1, to=4-1]
    \arrow["{\sim}", from=3-2, to=4-2]
    \arrow["{-{\phi_i}_*\varphi_i}", from=4-1, to=4-2]
\end{tikzcd}\]
Hence we obtain the the commutative diagram 
\[\begin{tikzcd}
	\mathscr{E}_i & \mathscr{H}om_{\mathscr{O}_U}(\mathscr{E}_i^{\vee},\mathscr{O}_U) \\
	\mathscr{E}_i^{\vee}. & 
	\arrow["{\sim}", from=1-1, to=1-2]
	\arrow["{-\Phi_i}", from=1-1, to=2-1]
	\arrow["{\Phi_i^*}", from=1-2, to=2-1]
\end{tikzcd}\]
So $Q_i'$ is alternating.

Finally, we claim that $\theta_i'$ is anti-self-adjoint with respect to $Q_i'$. 
Note that $\Phi_i$ induces the isomorphism 
\[\Phi_i':{\phi_i}_*\hat{\tau}_i^*\mathscr{L}_i
    \xrightarrow{{\phi_i}_*\varphi_i}{\phi_i}_*\mathscr{H}om_{\mathscr{O}_{W_i}}(\mathscr{L}_i,\phi_i'^*L_i')\xrightarrow{\sim}\mathscr{H}om_{\mathscr{O}_{U}}({\phi_i}_*\mathscr{L}_i,\mathscr{O}_{U})=\mathscr{E}_i^{\vee}
\]
of ${\phi_i}_*\mathscr{O}_{W_i}$-modules, given ${\phi_i}_*\hat{\tau}_i^*\mathscr{L}_i$ a structure of ${\phi_i}_*\mathscr{O}_{W_i}$-module by ${\phi_i}_*\hat{\tau}_i^*\mathscr{O}_{W_i}\cong {\phi_i}_*\mathscr{O}_{W_i}$ and $\mathscr{E}_i^{\vee}$ a natural structure of ${\phi_i}_*\mathscr{O}_{W_i}$-module.
Since we have the isomorphism ${\phi_i}_*\hat{\tau}_i^\#:{\phi_i}_*\mathscr{O}_{W_i}\to {\phi_i}_*\hat{\tau}_{i*}\mathscr{O}_{W_i}\cong {\phi_i}_*\mathscr{O}_{W_i}$, this gives $E_i':={\phi_i}_*\mathscr{L}_i$ a structure of ${\phi_i}_*\mathscr{O}_{W_i}$-module.
One may verify that $E_i'\to {\phi_i}_*\hat{\tau}_i^*\mathscr{L}_i$ is a isomorphism of ${\phi_i}_*\mathscr{O}_{W_i}$-modules.
Then $\Phi_i:E_i'\to \mathscr{E}_i^{\vee}$ is a isomorphism of ${\phi_i}_*\mathscr{O}_{W_i}$-modules.
Hence $\theta_i'$ is anti-self-adjoint with respect to $Q_i'$. 
This completes the proof.
\end{proof}

\begin{lemma}\label{lemma: (E,theta,Q)}
There is $(\mathscr{E}_b={p_b}_*\mathscr{F}_b,\theta_b,Q_b)\in h_{X,Sp_{2n}}^{-1}(b)$, where  $\mathscr{F}_b$ is a maximal Cohen--Macaulay $\mathscr{O}_{X_b}$-module of generic rank $1$, $\theta_{b}:\mathscr{T}_X\to {p_b}_*\mathscr{O}_{X_{b}}\to \mathscr{E}nd_{\mathscr{O}_X}(\mathscr{E}_b)$, and $Q_b:\mathscr{E}_b\otimes \mathscr{E}_b\to \mathscr{O}_X$ is a perfect alternating bilinear form such that $\theta_b$ is anti-self-adjoint with respect to $Q_b$. 
\end{lemma}
\begin{proof}
For $i\in \mathcal{A}_2$, $\dim(Y_i\setminus W_i)=\dim(X\setminus U)=0$ and $Y_i$ is regular by Lemma \ref{lem: Y_i=C times C' sm}, then there is $(\mathscr{E}_i={\pi_i}_*\mathscr{F}_i,\theta_i,Q_i)$ by Lemma \ref{lemma: A_2' case}, such that $\mathscr{F}_i$ is a maximal Cohen--Macaulay $\mathscr{O}_{X_i}$-module of generic rank $1$,  $\theta_{i}:\mathscr{T}_X\to {p_b}_*\mathscr{O}_{X_{b}}\to {\pi_i}_*\mathscr{O}_{X_{i}}\to \mathscr{E}nd_{\mathscr{O}_X}(\mathscr{E}_i)$,  $Q_i:\mathscr{E}_i\otimes \mathscr{E}_i\to \mathscr{O}_X$ is a perfect alternating bilinear form such that $\theta_i$ is anti-self-adjoint with respect to $Q_i$, and $(\mathscr{F}_i|_{W_i},\mathscr{E}_i|_{U},\theta_i|_U,Q_i|_U)=(\mathscr{L}_i,{\phi_i}_*\mathscr{L}_i,\theta_i',Q_i')$. 
For $i\in \mathcal{A}_0\cup \mathcal{A}_1$, by Lemma \ref{lemma: A_0 case} and Lemma \ref{lemma: A_1 case}, there is $(\mathscr{E}_i={\pi_i}_*\mathscr{F}_i,\theta_i,Q_i)$, where $\mathscr{F}_i$ is a maximal Cohen--Macaulay $\mathscr{O}_{X_i}$-module of generic rank $1$,  $\theta_{i}:\mathscr{T}_X\to {p_b}_*\mathscr{O}_{X_{b}}\to {\pi_i}_*\mathscr{O}_{X_{i}}\to \mathscr{E}nd_{\mathscr{O}_X}(\mathscr{E}_i)$, and  $Q_i:\mathscr{E}_i\otimes \mathscr{E}_i\to \mathscr{O}_X$ is a perfect alternating bilinear form such that $\theta_i$ is anti-self-adjoint with respect to $Q_i$

Let $Z=\sqcup_{i\in \mathcal{A}_0\cup \mathcal{A}_1\cup \mathcal{A}_2}Y_i$, and $q_i:Y_i\to Z$ be the natural morphism for $i\in \mathcal{A}_0\cup \mathcal{A}_1\cup \mathcal{A}_2$.
Let $\mathscr{F}=\oplus_{i\in \mathcal{A}_0\cup \mathcal{A}_1\cup \mathcal{A}_2}{q_i}_*\mathscr{F}_i$. 
Then $\mathscr{F}$ is a maximal Cohen--Macaulay $\mathscr{O}_{Z}$-module of generic rank $1$.
Since the natural morphism $f:Z\to \bigsqcup_{1\leq i\leq r} X_i\to X_b$ is finite surjective, the $\mathscr{O}_{X_b}$-module $\mathscr{F}_b:=f_*\mathscr{F}$ is maximal Cohen--Macaulay of generic rank $1$ by Lemma \ref{lemma: A to B, E CM}.
Since $\pi_i=p_b\circ f\circ q_i$ for $i\in \mathcal{A}_0\cup \mathcal{A}_1\cup \mathcal{A}_2$, we see that $\mathscr{E}_b={\pi_{b}}_*\mathscr{F}_b\cong \oplus_{i\in \mathcal{A}_0\cup \mathcal{A}_1\cup \mathcal{A}_2}\mathscr{E}_i$ and $\theta_{b}:\mathscr{T}_X\to {p_b}_*\mathscr{O}_{X_{b}}\to \mathscr{E}nd_{\mathscr{O}_X}(\mathscr{E}_b)$ is induced by $\theta_i$ for all $i\in \mathcal{A}_0\cup \mathcal{A}_1\cup \mathcal{A}_2$, then $(\mathscr{E}_b,\theta_{b})\in h_{X,GL_{2n}}^{-1}(b)$ by Lemma \ref{lem: X_b of ChenNgo}.
Let $Q_b=\oplus_{i\in \mathcal{A}_0\cup \mathcal{A}_1\cup \mathcal{A}_2}Q_i$.
Then $Q_b:\mathscr{E}_b\otimes \mathscr{E}_b\to \mathscr{O}_X$ is a perfect alternating bilinear form such that $\theta_b$ is anti-self-adjoint with respect to $Q_b$.
Hence $(\mathscr{E}_b,\theta_b,Q_b)\in h_{X,Sp_{2n}}^{-1}(b)$.
\end{proof}

\begin{theorem}\label{thm: sd_Sp(2n) surj for X=C1 times C2}
     The morphism $sd_{X,Sp_{2n}}$ is surjective.
\end{theorem}
\begin{proof}
For any $b\in \mathscr{B}_{X,Sp_{2n}}(\mathbb{K})\subseteq \mathscr{B}_{X,GL_{2n}}(\mathbb{K})$, then there is a partition $\mu=(\mu_1^{\alpha_1},\dots,\mu_s^{\alpha_s})$ of $2n$ such that 
$0<\mu_1<\cdots<\mu_s,0<\alpha_i,\sum_{i=1}^s\mu_i\alpha_i=n$ and $b(\eta)\in Sym^{2n}_{\mu}(T_X^*/X)$. 
By Lemma \ref{lem: decomposition in SS}, there exist spectral data $$b_i: X\to Sym^{\alpha_i}(T_X^*/X),1\leq i\leq s,$$ such that $b_i(\eta)\in Sym_{(1^{\alpha_i})}^{\alpha_i}(T_X^*/X)$ for $1\leq i\leq s$ and $b$ is
the following composition of morphisms
\[
X\xrightarrow{(b_1,\dots,b_s)} Sym^{\alpha_1}(T_X^*/X)\times_X\cdots\times_X Sym^{\alpha_s}(T_X^*/X)\xrightarrow{\tau_{\mu}} Sym^{2n}(T_X^*/X).
\]
Since $b\in \mathscr{B}_{X,Sp_{2n}}(\mathbb{K})$, we have $b(X)\subseteq Sym^{2n,Sp}(T_X^*/X)$.
Let $\mathcal{M}=\{i\mid 1\leq i\leq s, 2\nmid \alpha_i\}$.

If $\mathcal{M}=\emptyset$, then we may verify that $b_i\in \mathscr{B}_{X,Sp_{\alpha_i}}(\mathbb{K})$. Thus, there is $(\mathscr{E}_{b_i},\theta_{b_i},Q_{b_i})\in h_{X,Sp_{\alpha_i}}^{-1}(b_i)$ for any $1\leq i\leq s$ by Lemma \ref{lemma: (E,theta,Q)}. 
Let $\mathscr{E}_b=\oplus_{1\leq i\leq s}\mathscr{E}_{b_i}^{\oplus\mu_i}$, $\theta_b=\oplus_{1\leq i\leq s}\theta_{b_i}^{\oplus\mu_i}$ and $Q_b=\oplus_{1\leq i\leq s}Q_{b_i}^{\oplus\mu_i}$.
By Lemma \ref{lem: decomposition in SS}, $(\mathscr{E}_b,\theta_b)\in h_{X,GL_{2n}}^{-1}(b)$.
We may verify that $Q_b:\mathscr{E}_b\otimes \mathscr{E}_b\to \mathscr{O}_X$ is a perfect alternating bilinear form such that $\theta_b$ is anti-self-adjoint with respect to $Q_b$. Hence $(\mathscr{E}_b,\theta_b,Q_b)\in h_{X,Sp_{2n}}^{-1}(b)$.

If $m\in \mathcal{M}\neq \emptyset$, then  $\#\mathcal{M}= 1$ and $2\mid \mu_m$ by the definitions of $b$ and $\mu$.
We may verify that $b_i\in \mathscr{B}_{X,Sp_{\alpha_i}}(\mathbb{K})$ for any $1\leq i\neq m\leq s$, and there is $b_m'\in \mathscr{B}_{X,Sp_{\alpha_m-1}}(\mathbb{K})$ such that $b_m'(\eta)\in Sym_{(1^{\alpha_m-1})}^{\alpha_m-1}(T_X^*/X)$ and $b_m$ is
the following composition of morphisms
\[
X\xrightarrow{(b_m',0)} Sym^{\alpha_m-1}(T_X^*/X)\times_X Sym^{1}(T_X^*/X)\to Sym^{\alpha_m}(T_X^*/X).
\]
By Lemma \ref{lemma: (E,theta,Q)}, there is $(\mathscr{E}_{b_i},\theta_{b_i},Q_{b_i})\in h_{X,Sp_{\alpha_i}}^{-1}(b_i)$ for any $1\leq i\neq m\leq s$, and there is $(\mathscr{E}_{b_m'},\theta_{b_m'},Q_{b_m'})\in h_{X,Sp_{\alpha_m-1}}^{-1}(b_m')$.
Since $2\mid \mu_m$, there is a perfect alternating bilinear form $Q_0$ on $\mathscr{F}:=\mathscr{O}_X^{\oplus\mu_m}$.  
Let $\mathscr{E}_b=(\oplus_{1\leq i\neq m\leq s}\mathscr{E}_{b_i}^{\oplus\mu_i})\oplus(\mathscr{E}_{b_m'}^{\oplus\mu_m})\oplus\mathscr{F}$, $\theta_b=(\oplus_{1\leq i\neq m\leq s}\theta_{b_i}^{\oplus\mu_i})\oplus(\theta_{b_m'}^{\oplus\mu_m})\oplus 0$ and $Q_b=(\oplus_{1\leq i\neq m\leq s}Q_{b_i}^{\oplus\mu_i})\oplus(Q_{b_m'}^{\oplus\mu_m})\oplus Q_0$.
By Lemma \ref{lem: decomposition in SS}, $(\mathscr{E}_b,\theta_b)\in h_{X,GL_{2n}}^{-1}(b)$.
We may verify that $Q_b:\mathscr{E}_b\otimes \mathscr{E}_b\to \mathscr{O}_X$ is a perfect alternating bilinear form such that $\theta_b$ is anti-self-adjoint with respect to $Q_b$. Hence $(\mathscr{E}_b,\theta_b,Q_b)\in h_{X,Sp_{2n}}^{-1}(b)$.

\end{proof}

\bibliographystyle{amsalpha}\bibliography{Math}

@article{BNR89,
  title={Spectral curves and the generalised theta divisor.},
  author={Ramanan, S and Narasimhan, MS and Beauville, A},
  journal={Journal f{\"u}r die reine und angewandte Mathematik},
  volume={394},
  pages={169--179},
  year={1989}
}

@book{serre2000local,
  title={Local algebra},
  author={Serre, Jean-Pierre},
  year={2000},
  publisher={Springer Science \& Business Media}
}

@book{fulton2013intersection,
  title={Intersection theory},
  author={Fulton, William},
  volume={2},
  year={2013},
  publisher={Springer Science \& Business Media}
}

@article{chen2020hitchin,
author = {T. H. Chen and B. C. Ng{\^o}},
title = {{On the Hitchin morphism for higher-dimensional varieties}},
volume = {169},
journal = {Duke Mathematical Journal},
number = {10},
publisher = {Duke University Press},
pages = {1971 -- 2004},
year = {2020},
doi = {10.1215/00127094-2019-0085},
URL = {https://doi.org/10.1215/00127094-2019-0085}
}

@article{song2023image,
    author = {Song, Lei and Sun, Hao},
    title = {{On the image of Hitchin morphism for algebraic surfaces: The case ${\rm GL}_n$}},
    journal = {International Mathematics Research Notices},
    volume = {2024},
    number = {1},
    pages = {492-514},
    year = {2023},
    month = {03},
    doi = {10.1093/imrn/rnad043},
    url = {https://doi.org/10.1093/imrn/rnad043},
}

@article{li2024functions,
  title={Functions on the commuting stack via {Langlands} duality},
  author={Li, Penghui and Nadler, David and Yun, Zhiwei},
  journal={Annals of Mathematics},
  volume={200},
  number={2},
  pages={609--748},
  year={2024},
  publisher={Department of Mathematics, Princeton University Princeton, New Jersey, USA}
}

@BOOK{hartshorne1977AG,
    author={Hartshorne, Robin},
    title = {Algebraic Geometry},
    publisher = {Springer-Verlag},
    year = {1977},
    volume = {52},
    series= {Graduate Texts in Mathematics}
}

@article{Hit87,
  title={Stable bundles and integrable systems},
  author={Hitchin, N},
  journal={Duke mathematical journal},
  volume={54},
  number={1},
  pages={91--114},
  year={1987}
}

@book{gortz2020algebraic,
  title={{Algebraic geometry I: schemes}},
  author={G{\"o}rtz, Ulrich and Wedhorn, Torsten},
  year={2020},
  publisher={Springer}
}

@article{hartshorne1980stable,
  title={Stable reflexive sheaves},
  author={Hartshorne, Robin},
  journal={Mathematische annalen},
  volume={254},
  number={2},
  pages={121--176},
  year={1980},
  publisher={Springer}
}

@article{HL24,
  title={On the spectral variety for rank two {H}iggs bundles},
  author={He, Siqi and Liu, Jie},
  journal={Proceedings of the London Mathematical Society},
  volume={129},
  number={5},
  pages={e70004},
  year={2024},
  publisher={Wiley Online Library}
}

@book{liu2006algebraic,
  title={Algebraic Geometry and Arithmetic Curves},
  author={Liu, Qing and Erne, Reinie},
  volume={6},
  year={2006},
  publisher={Oxford University Press}
}

@article{hausel2012prym,
  title={Prym varieties of spectral covers},
  author={Hausel, Tam{\'a}s and Pauly, Christian},
  journal={Geometry \& Topology},
  volume={16},
  number={3},
  pages={1609--1638},
  year={2012},
  publisher={Mathematical Sciences Publishers}
}

@article{Nit91,
  title={Moduli space of semistable pairs on a curve},
  author={Nitsure, Nitin},
  journal={Proceedings of the London Mathematical Society},
  volume={3},
  number={2},
  pages={275--300},
  year={1991},
  publisher={Oxford University Press}
}

@article{Sim94,
  title={{Moduli of representations of the fundamental group of a smooth projective variety. II}},
  author={Simpson, Carlos T},
  journal={Publications math{\'e}matiques de l'IH{\'E}S},
  volume={80},
  number={1},
  pages={5--79},
  year={1994},
  publisher={Springer Science and Business Media LLC}
}

@article{huynh2026hitchin,
  title={The {H}itchin morphism for certain surfaces fibered over a curve},
  author={Huynh, Matthew},
  journal={Pure and Applied Mathematics Quarterly},
  volume={21},
  number={6},
  pages={2257--2286},
  year={2026},
  publisher={International Press of Boston}
}

@article{patel2026stratifying,
  title={Stratifying moduli spaces of {H}iggs bundles and the {H}itchin morphism},
  author={Patel, Aryaman and Weissmann, Dario},
  journal={arXiv preprint arXiv:2601.08597},
  year={2026}
}

@article{patel2026hitchin,
  title={The {H}itchin morphism for {K}-trivial varieties},
  author={Patel, Aryaman and Weissmann, Dario},
  journal={arXiv preprint arXiv:2604.03217},
  year={2026}
}

\end{document}